\documentclass[11pt,a4paper]{article}

\usepackage[T1]{fontenc}
\usepackage[utf8]{inputenc}
\usepackage{lmodern}
\usepackage[margin=25mm]{geometry}
\usepackage{amsmath,amsthm,amsfonts,amssymb}
\usepackage{booktabs}
\usepackage{microtype}
\usepackage{authblk}
\usepackage{cite}
\usepackage[hidelinks]{hyperref}

\numberwithin{equation}{section}
\newtheorem{theorem}{Theorem}[section]
\newtheorem{lemma}[theorem]{Lemma}
\newtheorem{corollary}[theorem]{Corollary}
\theoremstyle{definition}
\newtheorem{assumption}[theorem]{Assumption}
\theoremstyle{remark}
\newtheorem{remark}[theorem]{Remark}

\allowdisplaybreaks
\newcommand{\norm}[1]{\lVert #1\rVert}
\newcommand{\seminorm}[1]{\lvert #1\rvert}

\title{Sharp Pointwise Error Estimates of Finite Element Methods
for the Biharmonic Equation}
\author[1]{Buyang Li}
\author[2]{Zhenglei Li}
\author[2]{Qigang Liang}
\author[2]{Xuejun Xu\thanks{Corresponding author: \texttt{xuxj@tongji.edu.cn}.}}
\affil[1]{Department of Applied Mathematics, The Hong Kong Polytechnic University,\\
Hong Kong, China\\
\texttt{buyang.li@polyu.edu.hk}}
\affil[2]{School of Mathematical Sciences, Tongji University, Shanghai 200092, China;\\
and Key Laboratory of Intelligent Computing and Applications\\
(Tongji University), Ministry of Education\\
\texttt{2211186@tongji.edu.cn}, \texttt{qigang\_liang@tongji.edu.cn},
\texttt{xuxj@tongji.edu.cn}}
\date{}
\hypersetup{
  pdftitle={Sharp Pointwise Error Estimates of Finite Element Methods for the Biharmonic Equation},
  pdfauthor={Buyang Li, Zhenglei Li, Qigang Liang, Xuejun Xu},
  pdfkeywords={biharmonic equation, finite element method, pointwise error estimate}
}

\begin{document}
\maketitle
\begin{abstract}
We prove sharp pointwise error estimates for finite element approximations of
the clamped biharmonic equation on convex polygonal domains.  For
\(C^1\)-conforming elements, we establish \(W^{2,p}\)-stability for all
\(2\le p\le \infty\), together with best-approximation type estimates in
\(W^{1,\infty}\) and \(L^\infty\).  For the nonconforming Morley element, we
derive optimal estimates in \(W^{2,p}\), \(W^{1,\infty}\), and \(L^\infty\);
in particular, the \(L^\infty\)-error is of order \(O(h^2)\) without an
additional logarithmic factor.  The analysis is based on three ingredients:
local energy estimates for conforming and nonconforming plate elements, a
removal of the local \(L^2\)-error term in these estimates, and H\"older
estimates for derivatives of the Green's function.  These tools are combined
with a dyadic decomposition argument for regularized Green's functions.  The
results remove logarithmic losses present in earlier pointwise estimates for
fourth-order problems and yield sharp rates for classical plate elements such
as the Argyris, Bogner--Fox--Schmit (BFS),  Hsieh--Clough--Tocher (HCT), and Morley elements.
\end{abstract}

\par

\par\medskip

\setlength\abovedisplayskip{5pt}
\setlength\belowdisplayskip{5pt}

\section{Introduction}

Let \(\Omega\subset\mathbb{R}^2\) be a bounded convex polygon.  We consider the
clamped biharmonic problem
\begin{align}\label{definition model problem}
	\begin{cases}
		\Delta^2 u=f & \text{in } \Omega,\\
		u=\dfrac{\partial u}{\partial n}=0 & \text{on } \partial\Omega,
	\end{cases}
\end{align}
where \(f\in L^2(\Omega)\) and \(n\) denotes the outward unit normal vector on
\(\partial\Omega\).  This is the standard clamped Kirchhoff--Love plate model. Its finite element approximation is
well understood in energy norms, but sharp pointwise estimates are considerably
more delicate, especially on polygonal domains where corner singularities limit
global regularity.

For second-order elliptic problems, maximum-norm and pointwise finite element
estimates have a long history.  Interior estimates and maximum-norm estimates
were developed in the seminal works of Nitsche--Schatz, Bramble--Nitsche--Schatz,
and Schatz--Wahlbin \cite{NitscheSchatz1974,BrambleNitscheSchatz1975,
SchatzWahlbin1977,SchatzWahlbin1978}.  Global \(L^\infty\)-type estimates on
irregular meshes were obtained by Scott \cite{Scott}, and the removal of
unnecessary logarithmic factors for Ritz projections was achieved in the
work of Rannacher and Scott \cite{MR645661}.  These ideas have continued to
play an important role in modern pointwise finite element analysis, including
local energy techniques \cite{MR2728969}, maximum-norm stability on locally
refined meshes \cite{MR4435940,MR4050547}, pointwise estimates for parabolic
problems \cite{MR3498514}, and recent pointwise gradient estimates for Ritz
projections \cite{MR4748207}.

The fourth-order case is much less complete.  Compared with second-order
problems, the biharmonic equation involves stronger boundary conditions,
\(C^1\)-conforming or nonconforming discretizations, and Green's functions with
more singular higher derivatives.  Early pointwise results for plate problems
include Rannacher's estimates for mixed and nonconforming methods
\cite{MR555385} and Wang's \(L^\infty\)-estimates for conforming and
nonconforming finite elements \cite{MR1252774,MR1229633}.  These results are
nearly optimal in several cases, but logarithmic factors remain in the
\(W^{1,\infty}\)- or \(L^\infty\)-bounds.  More recently, Leykekhman
\cite{MR4166452} developed pointwise best-approximation estimates for the
\(C^0\) interior penalty approximation of biharmonic problems on convex
polygons.  That work introduced a powerful approach based on dyadic
decompositions, local energy estimates, and pointwise Green's function bounds.
However, the resulting estimates still contain logarithmic factors, partly
because the local energy analysis retains lower-order error terms that are
then controlled globally.

The goal of this paper is to prove sharp pointwise error estimates for finite
element approximations of \eqref{definition model problem}, both for
\(C^1\)-conforming elements and for nonconforming lower-order elements such as
the Morley element.  The main novelty is not merely the final convergence rates,
but the mechanism by which logarithmic losses are avoided.  Our analysis has
three key ingredients.

First, we prove local energy estimates for a class of conforming and
nonconforming finite element spaces satisfying an abstract superapproximation
assumption.  In contrast with previous local estimates for fourth-order
problems, we eliminate the local \(L^2\)-term from the final bound and retain
only an \(H^1\)-term.  This distinction is crucial: an \(L^2\)-term would require
a duality argument involving \(H^4\)-regularity, which is generally unavailable
on polygonal domains, whereas the \(H^1\)-term can be handled using only
\(H^3\)-type regularity.

Second, we establish H\"older estimates for derivatives of the continuous
Green's function of the clamped biharmonic operator.  These estimates are
derived from the pointwise Green's function bounds of Maz'ya and Rossmann
\cite{MR1143406}.  The positive H\"older exponent obtained in this step is
essential for summing the dyadic estimates without producing logarithmic
growth.

Third, we combine the local energy estimates and the Green's function
H\"older estimates through a dyadic decomposition.  This yields a uniform
\(W^{2,1}\)-estimate for the finite element error of a regularized derivative
Green's function.  This estimate is the core technical result behind the sharp
\(W^{2,p}\), \(W^{1,\infty}\), and \(L^\infty\) finite element bounds.

We now state the main results.  The finite element spaces, the broken bilinear
form, and the discrete solution \(u_h\) are defined precisely in
Section~\ref{sec:fem-setting}. Let $\mu$ denote the largest integer such that every polynomial of degree at most $\mu$ is reproduced exactly by the corresponding (quasi-)interpolation operator.  In particular, all finite element solutions
appearing below are understood as solutions of the numerical scheme
\eqref{eq:discrete-scheme}.  Throughout the paper, \(A\lesssim B\) means
that \(A\leq CB\), where the constant \(C\) is independent of the mesh size
\(h\), and \(A\simeq B\) means \(A\lesssim B\lesssim A\). 
\begin{theorem}[Conforming elements]\label{theorem conforming main results}
Let \(u\in H_0^2(\Omega)\) solve \eqref{definition model problem}, and let \(u_h\in \mathring{V}_h\) be the conforming finite element solution determined by \eqref{eq:discrete-scheme}. Assume that \(V_h\) satisfies Assumption~\ref{assumption}. 

\begin{subequations}
If \(\mu\ge 2\), then
\begin{align}\label{conforming W2,p}
    \seminorm{u_h}_{W^{2,p}(\Omega)}
    \lesssim
    \seminorm{u}_{W^{2,p}(\Omega)},
    \qquad 2\le p\le\infty.
\end{align}

If \(\mu\ge 3\), then
\begin{align}
\label{conforming W1,infty}
    \seminorm{u-u_h}_{W^{1,\infty}(\Omega)}
    &\lesssim
    \inf_{v_h\in\mathring V_h}
    \seminorm{u-v_h}_{W^{1,\infty}(\Omega)}
    +
    h\inf_{v_h\in\mathring V_h}
    \seminorm{u-v_h}_{W^{2,\infty}(\Omega)}, \\
\label{conforming Linfty-log}
    \norm{u-u_h}_{L^\infty(\Omega)}
    &\lesssim
    \inf_{v_h\in\mathring V_h}
    \norm{u-v_h}_{L^\infty(\Omega)}
    +
    h^2|\ln h|
    \inf_{v_h\in\mathring V_h}
    \seminorm{u-v_h}_{W^{2,\infty}(\Omega)}.
\end{align}

If \(\mu\ge 4\), the factor \(|\ln h|\) in the last estimate can be removed, i.e.,
\begin{align}\label{conforming Linfty}
     \norm{u-u_h}_{L^\infty(\Omega)}
    &\lesssim
    \inf_{v_h\in\mathring V_h}
    \norm{u-v_h}_{L^\infty(\Omega)}
    +
    h^2
    \inf_{v_h\in\mathring V_h}
    \seminorm{u-v_h}_{W^{2,\infty}(\Omega)}.
\end{align}
\end{subequations}
\end{theorem}

Theorem~\ref{theorem conforming main results} sharpens the earlier maximum-norm estimates in \cite{MR1252774}. In particular, the \(W^{2,p}\)-stability \eqref{conforming W2,p}, including the endpoint \(p=\infty\), yields a quasi-optimal \(W^{2,p}\)-error estimate in terms of the best approximation error without logarithmic deterioration. Moreover, \eqref{conforming W1,infty} removes the logarithmic loss in the \(W^{1,\infty}\)-estimate of \cite{MR1252774}. More significantly, the \(L^\infty\)-estimates \eqref{conforming Linfty-log} and \eqref{conforming Linfty} recover the missing algebraic order in \cite{MR1252774}: \eqref{conforming Linfty-log} gives the optimal algebraic rate up to a logarithmic factor, while for \(\mu\ge4\), \eqref{conforming Linfty} removes this logarithmic factor as well. Consequently, for sufficiently smooth solutions, the expected pointwise convergence rates of high-order conforming elements are recovered. For example, for the Argyris element, the rates are \(O(h^4)\) in \(W^{2,p}\), \(O(h^5)\) in \(W^{1,\infty}\), and \(O(h^6)\) in \(L^\infty\), with no additional logarithmic losses.

Our second main result concerns the nonconforming Morley element.  Here the
finite element solution is again defined by the same abstract scheme
\eqref{eq:discrete-scheme}, but with the Morley space and the broken bilinear
form \eqref{eq:discrete-bilinear-form}.  The energy error is only first order,
but the lower-order pointwise errors exhibit a sharper structure.

\begin{theorem}[Morley element]\label{theorem nonconforming main results}
Let \(V_h\) be the triangular Morley finite element space, and let
\(u_h\in\mathring V_h\) be the Morley finite element solution determined by \eqref{eq:discrete-scheme} with the broken bilinear form \eqref{eq:discrete-bilinear-form}. 

If \(2\le p\le\infty\), \(u\in W^{3,p}(\Omega)\), and \(f\in L^p(\Omega)\),
then
\begin{subequations}
\begin{align}
    \label{nonconforming W2,p}
    \seminorm{u-u_h}_{W^{2,p}(\Omega)}
    \lesssim
    h\left(
    \seminorm{u}_{W^{3,p}(\Omega)}
    +
    h\norm{f}_{L^p(\Omega)}
    \right).
\end{align}

If \(u\in W^{3,\infty}(\Omega)\) and \(f\in L^\infty(\Omega)\), then
\begin{align}\label{nonconforming W1,infty}
    \seminorm{u-u_h}_{W^{1,\infty}(\Omega)}
    \lesssim
    h^2|\ln h|
    \left(
    \seminorm{u}_{W^{3,\infty}(\Omega)}
    +
    h\norm{f}_{L^\infty(\Omega)}
    \right).
\end{align}

If \(2< p\le\infty\), \(u\in W^{3,p}(\Omega)\),
\(f\in L^p(\Omega)\),  then
\begin{align}\label{nonconforming Linfty}
    \norm{u-u_h}_{L^\infty(\Omega)}
    \lesssim
    h^2
    \left(
    \seminorm{u}_{W^{3,p}(\Omega)}
    +
    h\norm{f}_{L^p(\Omega)}
    \right).
\end{align}
\end{subequations}
\end{theorem}

The estimate \eqref{nonconforming Linfty} improves the earlier
\(L^\infty\)-theory for nonconforming biharmonic elements by removing the
logarithmic factor in the function-value error.  The remaining logarithm in
\eqref{nonconforming W1,infty} is caused by the low polynomial degree of the
Morley element and is analogous to the logarithmic behavior of linear finite
elements for second-order problems.  Since the solution of the clamped
biharmonic problem on a convex polygon has \(W^{3,2+\varepsilon}\)-regularity
for some \(\varepsilon>0\) depending only on the domain \cite{MR775683},
\eqref{nonconforming Linfty} gives the second-order \(L^\infty\)-bound essentially under
the assumption $f\in L^p(\Omega)$ for any $p>2$.  Combined with the lower bound of Hu and Shi
\cite{MR2988473} for lower-order plate elements, this shows the sharpness of
the Morley \(L^q\)-error:
\[
    \norm{u-u_h}_{L^q(\Omega)}\simeq h^2,
    \qquad 2\leq q\leq \infty,
\]
under the corresponding nondegeneracy assumptions.


For the reader's convenience, the interpolation parameter $\mu$ associated with each finite element space and the corresponding convergence rates in various norms are summarized in Table~\ref{tab:interpolation-rates}.

\begin{table}[htbp]
\caption{Interpolation parameters and convergence rates for several
finite element spaces in the \(W^{2,p}\), \(W^{1,\infty}\),
and \(L^\infty\) norms. The rates are stated for sufficiently smooth solutions and data.}
\label{tab:interpolation-rates}
\centering

\small
\setlength{\tabcolsep}{4pt}
\renewcommand{\arraystretch}{1.0}

\begin{tabular}{lcccc}
\toprule
Finite element & $\mu$ & $W^{2,p}$ & $W^{1,\infty}$ & $L^\infty$ \\
\midrule
Argyris & $5$ & $O(h^4)$ & $O(h^5)$ & $O(h^6)$ \\
Bell & $4$ & $O(h^3)$ & $O(h^4)$ & $O(h^5)$ \\
BFS & $3$ & $O(h^2)$ & $O(h^3)$ & $O\bigl(h^4|\ln h|\bigr)$ \\
HCT & $3$ & $O(h^2)$ & $O(h^3)$ & $O\bigl(h^4|\ln h|\bigr)$ \\
Morley & $2$ & $O(h)$ & $O\bigl(h^2|\ln h|\bigr)$ & $O(h^2)$ \\
\bottomrule
\end{tabular}
\end{table}

The rest of the paper is organized as follows.  Section~\ref{sec:fem-setting}
introduces the finite element spaces, the broken bilinear form, and the numerical scheme \eqref{eq:discrete-scheme}.  Section~\ref{sec:auxiliary-estimates}
develops the technical tools needed to analyze that scheme: superapproximation
estimates, local energy estimates, H\"older estimates for Green's functions,
and a priori estimates for regularized Green's functions.  The local estimates
in Section~\ref{subsec:local-energy} are the bridge between the discrete scheme
of Section~\ref{sec:fem-setting} and the dyadic Green's function argument used
later.  Section~\ref{sec:pointwise-estimates} proves the sharp pointwise estimates for conforming finite
elements and the nonconforming Morley element. Section~\ref{conclusion} concludes the paper.  Appendix~\ref{supp:assumption} verifies the finite element hypotheses, Appendix~\ref{supp:regularity} establishes the required $W^{4,p}$-regularity estimate, and Appendix~\ref{supp:Morley} provides the detailed consistency error estimates for the Morley element.

\section{Finite element setting}\label{sec:fem-setting}

This section fixes the notation used in the rest of the paper and gives the
precise finite element scheme referred to in the main results.  The pointwise
analysis below treats both \(C^1\)-conforming plate elements and
nonconforming plate elements.  Therefore, unless explicitly stated otherwise,
Sobolev norms and seminorms are understood in the broken sense with respect to
the mesh.

\subsection{Meshes, broken spaces, and polynomial spaces}

Let \(\mathcal T_h\) be a shape-regular and quasi-uniform
partition of \(\Omega\) into triangles or rectangles, according
to the finite element under consideration,  and let \(h:=\max_{T\in\mathcal{T}_h} h_T\), where \(h_T\) is the
diameter of \(T\).  For a subdomain \(U\subset \Omega\), define
$\mathcal{T}_h(U):=\{T\in\mathcal{T}_h:\ T\subset U\}.$
For simplicity, all subdomains used in the local estimates are assumed to be
unions of elements.  This convention avoids inessential boundary-layer
notation and does not affect the estimates.

For \(m\geq 0\), we define the broken Sobolev space
\begin{equation}\label{eq:broken-Hm-space}
    H^m(U,\mathcal{T}_h)
    :=
    \{v\in L^2(U):\ v|_T\in H^m(T)
    \text{ for all }T\in\mathcal{T}_h(U)\}.
\end{equation}
For \(1\leq p<\infty\), the broken \(W^{m,p}\)-norm and seminorm are
\begin{subequations}\label{eq:broken-Wmp-defs}
\begin{align*}
    \norm{v}_{W^{m,p}(U,\mathcal{T}_h)}
    &:=
    \left(
    \sum_{T\in\mathcal{T}_h(U)}
    \norm{v}_{W^{m,p}(T)}^p
    \right)^{\frac1p}, \ 
    \seminorm{v}_{W^{m,p}(U,\mathcal{T}_h)}
    :=
    \left(
    \sum_{T\in\mathcal{T}_h(U)}
    \seminorm{v}_{W^{m,p}(T)}^p
    \right)^{\frac1p}. 
\end{align*}
\end{subequations}
For \(p=\infty\), we use the corresponding maximum over elements:
\begin{subequations}\
\begin{align*}
    \norm{v}_{W^{m,\infty}(U,\mathcal{T}_h)}
    &:=
    \max_{T\in\mathcal{T}_h(U)}
    \norm{v}_{W^{m,\infty}(T)},\quad 
    \seminorm{v}_{W^{m,\infty}(U,\mathcal{T}_h)}
    :=
    \max_{T\in\mathcal{T}_h(U)}
    \seminorm{v}_{W^{m,\infty}(T)}. 
\end{align*}
\end{subequations}
When no confusion can arise, we suppress the symbol \(\mathcal{T}_h\).  In
particular, for \(m\geq 1\), we write
\begin{equation}\label{eq:broken-Hhm-seminorm}
    \seminorm{v}_{H_h^m(U)}
    :=
    \seminorm{v}_{H^m(U,\mathcal{T}_h)},
    \qquad
    \norm{v}_{H_h^m(U)}
    :=
    \norm{v}_{H^m(U,\mathcal{T}_h)}.
\end{equation}
For \(m=0\), \(\seminorm{v}_{H_h^0(U)}\) is understood as
\(\norm{v}_{L^2(U)}\).

For \(k\geq0\), let \(P_k(T)\) denote the space of polynomials
of total degree at most \(k\) when \(T\) is a triangle, and
of degree at most \(k\) in each coordinate when \(T\) is a
rectangle. We define the corresponding broken polynomial
space by
\begin{equation*}
    P_k(U,\mathcal{T}_h)
    :=
    \{v\in L^2(U):\ v|_T\in P_k(T)
    \text{ for all }T\in\mathcal{T}_h(U)\}.
\end{equation*}
When \(U=\Omega\), we simply write \(P_k(\mathcal{T}_h)\).  For
\(v\in L^1(U)\), the average of \(v\) over \(U\) is denoted by  $P_U^0 v
    :=
    \frac{1}{|U|}\int_U v\,\mathrm{d}x$. 
We also use \(I_h^k\) to denote the standard nodal interpolation
operator into \(P_k(\mathcal T_h)\cap C^0(\overline\Omega)\).

Throughout the paper, \(D\) denotes a first-order derivative, \(D^2\) denotes
the Hessian, and \(D^m\) denotes derivatives of order \(m\).  We write
\(\partial_i\) for the derivative with respect to the \(i\)-th coordinate.

\subsection{Continuous weak formulation}

The weak formulation of \eqref{definition model problem} is:
find \(u\in H_0^2(\Omega)\) such that
\begin{equation}\label{eq:continuous-weak-problem}
    a(u,v)=(f,v)
    \qquad \forall v\in V:=H_0^2(\Omega),
\end{equation}
where the bilinear form \(a(\cdot,\cdot)\) is defined by
\begin{equation}\label{eq:continuous-bilinear-form}
    a(w,v)
    :=
    \int_\Omega D^2w:D^2v\,\mathrm{d}x,
\end{equation}
and $(\cdot,\cdot)_U$  denotes the standard $L^2$ inner product on $U\subset\Omega$ (if $U=\Omega$, we omit the subscript $U$).
Here $A:B:=\sum_{i,j=1}^2 A_{ij}B_{ij}$ for two \(2\times2\) matrices \(A\) and \(B\).  On $V\times V$, the
bilinear form \(a(\cdot,\cdot)\) is coercive and induces a norm equivalent to
the \(H^2\)-seminorm.

\subsection{Finite element spaces and the numerical scheme}

Let \(V_h\) denote either a \(C^1\)-conforming finite element space or a
nonconforming plate element space, such as the triangular Morley element.  For
a subdomain \(U\subset\Omega\), \(V_h(U)\) denotes the restriction of \(V_h\)
to \(U\), and \(\mathring{V}_h(U)\) denotes the corresponding local finite
element space with homogeneous boundary degrees of freedom on \(\partial U\).
When \(U=\Omega\), we use the abbreviations $V_h:=V_h(\Omega)$ and    $\mathring{V}_h:=\mathring{V}_h(\Omega)$.
For conforming elements, \(\mathring{V}_h\subset H_0^2(\Omega)\).  For
nonconforming elements, \(\mathring{V}_h\) is the usual nonconforming analogue
of \(H_0^2(\Omega)\), with the clamped boundary conditions imposed through the
degrees of freedom.

For \(w,v\in H^2(\Omega,\mathcal{T}_h)\), define the broken bilinear form
\begin{equation}\label{eq:discrete-bilinear-form}
    a_h(w,v)
    :=
    \sum_{T\in\mathcal{T}_h}
    \int_T D^2w:D^2v\,\mathrm{d}x.
\end{equation}
The finite element approximation of \eqref{definition model problem} is the
following numerical scheme: Find  $u_h\in \mathring{V}_h$ such that
\begin{equation}\label{eq:discrete-scheme}
\begin{aligned}
    a_h(u_h,v_h)=(f,v_h)
    \qquad \forall v_h\in \mathring{V}_h.
\end{aligned}
\end{equation}
For conforming elements, \(a_h=a\) on \(V_h\times V_h\), and
\eqref{eq:discrete-scheme} is the standard Galerkin method.  For
nonconforming elements, \(a_h\) is the natural broken energy form.

The local analysis below uses a discrete negative seminorm associated with
\eqref{eq:discrete-bilinear-form}.  For
\(r\in H^2(U,\mathcal{T}_h)\), define
\begin{equation}\label{eq:local-negative-seminorm}
    \seminorm{r}_{H_h^{-2}(U)}
    :=
    \sup_{0\ne w_h\in \mathring{V}_h(U)}
    \frac{|a_h(r,w_h)|}
    {\seminorm{w_h}_{H_h^2(U)}}.
\end{equation}
This seminorm measures the residual of \(r\) against local discrete test
functions.  In the conforming case it is mainly a localization device, whereas
for nonconforming elements it also captures the local consistency error.

\section{Local energy estimates and estimates for the Green's function of the biharmonic problem}\label{sec:auxiliary-estimates}

This section develops two technical estimates. The first ingredient is a local energy estimate for the error measured in the
broken \(H^2\)-seminorm.  Its proof relies on a superapproximation property
and is formulated so that it applies simultaneously to conforming \(C^1\)
plate elements and to nonconforming elements such as the Morley element.

The second ingredient, developed in the following subsection, is a H\"older
estimate for derivatives of the Green's function of the clamped biharmonic
operator.  Together with the local energy estimate, this H\"older continuity
allows us to sum dyadic estimates for regularized Green's functions without
introducing unnecessary logarithmic factors.  Thus Section~\ref{subsec:local-energy}
provides the link between the abstract finite element discretization and the pointwise estimates proved later.

\subsection{Superapproximation and local energy estimates}\label{subsec:local-energy}

The purpose of this subsection is to localize the error equation associated
with \eqref{eq:discrete-scheme}.  We work on nested element patches and use
cut-off functions to separate the error on a small patch from the behavior of
the solution on a larger patch.

Let \(U_1\Subset  U_2\subset\Omega\) be two element patches.  We use
\(\eta\in C^\infty(\Omega)\) to denote a cut-off function satisfying
\begin{align}\label{pro of eta}
        \begin{aligned}
            &\eta \equiv 1 \quad \text{in } U_1,
            \quad 0\leq \eta \leq 1 \quad \text{in } U_2\setminus U_1,
            \quad \eta \equiv 0 \quad \text{in } \Omega\setminus U_2,\\            &\norm{D^k\eta}_{L^\infty(\Omega)}\lesssim d^{-k},
            \qquad k=0,1,2,\ldots,
        \end{aligned}
\end{align}
where $d:=
    \operatorname{dist}
    \bigl(\partial U_1\setminus\partial\Omega,\,
          \partial U_2\setminus\partial\Omega\bigr).$
In the local arguments below, we always assume that \(d\geq \kappa h\), where
\(\kappa>0\) is a sufficiently large constant independent of \(h\).

The following assumption collects precisely the finite element properties
needed to localize the scheme \eqref{eq:discrete-scheme}.  The interpolation
property in (A1) controls approximation errors, (A2) provides the required superapproximation property, and (A3) supplies a conforming companion needed in the lower-order part
of the local estimate.

\begin{assumption}[Finite element hypotheses]\label{assumption}
Assume that $v_h$ is continuous at all vertices for any $v_h\in V_h$. Moreover, we assume that \(V_h\) satisfies the following properties. Recall that $\mu\geq 2$ is the maximal degree of polynomials exactly reproduced by the corresponding interpolation operator.

\begin{enumerate}
\item[(A1)] Denote by $\omega_T$  the element patch of $T$. There exists a boundary-preserving (quasi-) interpolation operator $\Pi_h$, with the projection property, satisfying
for any \(T\in\mathcal{T}_h\), any
\(v\in W^{s,p}(\omega_T)\), \(1\le p\le \infty\), $3\leq s \leq \mu+1$ and \(m=0,1,2\),
\begin{equation}\label{approximation of interpolation}
  \seminorm{v-\Pi_hv}_{W^{m,p}(T)}
  \lesssim
  h^{s-m}\seminorm{v}_{W^{s,p}(\omega_T)}.
\end{equation}
Moreover, there also holds that
\begin{equation}\label{Holder approximation of interpolation}
 \seminorm{v-\Pi_hv}_{W^{2,\infty}(T)}
  \lesssim  h^{t}\seminorm{v}_{W^{2+t,\infty}(\omega_T)} \quad \forall v\in W^{2+t,\infty}(\omega_T),\, 0\leq t\leq 1.    
\end{equation}

\item[(A2)] Let \(\eta\) be as in \eqref{pro of eta}.  For every
\(v_h\in \mathring{V}_h\), there exists \(w_h\in \mathring{V}_h(U_2)\) such
that
\begin{equation}\label{assumption pi}
    \seminorm{\eta^4v_h-w_h}_{H_h^2(U_2)}
    \lesssim
    h\Bigl(
    d^{-1}\seminorm{\eta^2 v_h}_{H_h^2(U_2)}
    +d^{-2}\seminorm{v_h}_{H_h^1(U_2)}
    +d^{-3}\norm{v_h}_{L^2(U_2)}
    \Bigr).
\end{equation}
If \(U_2\cap\partial\Omega=\varnothing\), we assume that
\eqref{assumption pi} is also valid for all \(v_h\in V_h\).

\item[(A3)] For any \(v_h\in \mathring{V}_h(U)\), there exists
\(v\in H_0^2(U)\) such that
\begin{equation}\label{assumption 4}
    \seminorm{v-v_h}_{H_h^j(U)}
    \lesssim
    h^{2-j}\seminorm{v_h}_{H_h^2(U)},
    \qquad j=0,1,2.
\end{equation}
\end{enumerate}
\end{assumption}
Further discussion and verification of the above assumptions are provided in Appendix~\ref{supp:assumption}.

\par We now derive the local energy estimate.  The estimate is stated for a general
pair \(v\in H^2(\Omega,\mathcal{T}_h)\) and \(v_h\in V_h\).  Its main
application is obtained by taking \(v=u-\chi_h\) and \(v_h=u_h-\chi_h\), where
\(u\) solves \eqref{eq:continuous-weak-problem}, \(u_h\) is the approximate solution, and \(\chi_h\in\mathring{V}_h\) is an
arbitrary comparison function.  In this way the abstract local estimate below
becomes a local estimate for the finite element error.

\begin{theorem}[Local energy estimate]\label{theorem local energy estimate}
Let \(U\subset G_0\subset G\) be element patches such that
\[
    \operatorname{diam}(U)\simeq \operatorname{diam}(G_0)
    \simeq \operatorname{diam}(G)\simeq d,
\]
\[
    \operatorname{dist}
    \bigl(\partial U\setminus\partial\Omega,\,
          \partial G_0\setminus\partial\Omega\bigr)
    \simeq d \quad \text{and} \quad \operatorname{dist}
    \bigl(\partial G_0\setminus\partial\Omega,\,
          \partial G\setminus\partial\Omega\bigr)
    \simeq d.
\]
Assume \(d\ge \kappa h\), where \(\kappa\) is sufficiently large.  Then, under Assumption~\ref{assumption}, for
any \(v\in H^2(\Omega,\mathcal{T}_h)\) and \(v_h\in\mathring{V}_h\),
\begin{equation}\label{local energy estimate}
\begin{aligned}
        \seminorm{v-v_h}_{H_h^2(U)}
       & \lesssim
        \sum_{j=0}^2 d^{j-2}\seminorm{v}_{H_h^j(G)}
        + d^{-1}\seminorm{v-v_h}_{H_h^1(G_0)}
        + d^{-2}\norm{v-v_h}_{L^2(G_0)}
        \\
        &\quad
        + \seminorm{v-v_h}_{H_h^{-2}(G_0)}.
\end{aligned}
\end{equation}
If \(G_0\cap\partial\Omega=\varnothing\), then
\eqref{local energy estimate} holds for all \(v_h\in V_h\).
\end{theorem}

\begin{proof}
We prove the estimate for \(v_h\in\mathring{V}_h\).  The proof in the interior
case is identical, using the interior version of the superapproximation
estimate.

Let \(\eta\) be a cut-off function such that \(\eta\equiv 1\) on \(G_0\) and
\(\eta\equiv 0\) outside \(G\).  Define the local energy projection
\(P_h^G w\in\mathring{V}_h(G)\) by
\[
    a_h(P_h^G w,\chi_h)
    =
    a_h(w,\chi_h)
    \qquad \forall \chi_h\in\mathring{V}_h(G).
\]
Set $e_h:=P_h^G(\eta v)-v_h.$
Since \(\eta=1\) on \(U\), the triangle inequality and the stability of the
local energy projection give
\begin{align}\label{first step local energy}
\begin{aligned}
   \seminorm{v-v_h}_{H_h^2(U)}
   &\le
   \seminorm{\eta v-P_h^G(\eta v)}_{H_h^2(U)}
   +
   \seminorm{e_h}_{H_h^2(U)}
   \\
   &\lesssim
   \sum_{j=0}^2 d^{j-2}\seminorm{v}_{H_h^j(G)}
   +
   \seminorm{e_h}_{H_h^2(U)}.
\end{aligned}
\end{align}

It remains to estimate \(\seminorm{e_h}_{H_h^2(U)}\).  Let
\(\widetilde{\eta}\) be another cut-off function satisfying
\(\widetilde{\eta}\equiv 1\) on \(U\) and
\(\widetilde{\eta}\equiv 0\) outside \(G_0\).  Expanding
\(D^2(\widetilde{\eta}^2e_h)\) and using the bounds on
\(\widetilde{\eta}\), we obtain
\begin{equation}\label{cutoff energy inequality}
\begin{aligned}
    \seminorm{\widetilde{\eta}^2e_h}_{H_h^2(G_0)}^2
    \lesssim\;&
    \left|a_h(e_h,\widetilde{\eta}^4e_h)\right|
    +d^{-2}\seminorm{e_h}_{H_h^1(G_0)}^2
    +d^{-4}\norm{e_h}_{L^2(G_0)}^2.
\end{aligned}
\end{equation}
Indeed, the remainder terms produced by differentiating
\(\widetilde{\eta}\) are bounded by
\[
    d^{-2}\seminorm{e_h}_{H_h^1(G_0)}^2
    +d^{-4}\norm{e_h}_{L^2(G_0)}^2
    +d^{-2}\norm{e_h}_{L^2(G_0)}
      \seminorm{\widetilde{\eta}^2e_h}_{H_h^2(G_0)},
\]
and the last term is absorbed by Young's inequality.

By Assumption~(A2), there exists
\(w_h\in\mathring{V}_h(G_0)\) such that
\begin{equation}\label{superapproximation local energy}
   \seminorm{\widetilde{\eta}^4e_h-w_h}_{H_h^2(G_0)}
   \lesssim
   h\left(
   d^{-1}\seminorm{\widetilde{\eta}^2e_h}_{H_h^2(G_0)}
   +d^{-2}\seminorm{e_h}_{H_h^1(G_0)}
   +d^{-3}\norm{e_h}_{L^2(G_0)}
   \right).
\end{equation}
Since \(w_h\) is supported in \(G_0\), where \(\eta\equiv 1\), the definition
of \(P_h^G\) gives
\[
    a_h(e_h,w_h)
    =
    a_h(v-v_h,w_h).
\]
Therefore,
\[
    a_h(e_h,\widetilde{\eta}^4e_h)
    =
    a_h(e_h,\widetilde{\eta}^4e_h-w_h)
    +
    a_h(v-v_h,w_h).
\]
Using \eqref{superapproximation local energy}, inverse estimates on the support
of \(\widetilde{\eta}\), the definition of the negative seminorm, and
Young's inequality, we infer
\begin{equation}\label{sfinal improved}
    \seminorm{\widetilde{\eta}^2e_h}_{H_h^2(G_0)}
    \lesssim
    d^{-1}\seminorm{e_h}_{H_h^1(G_0)}
    +d^{-2}\norm{e_h}_{L^2(G_0)}
    +\seminorm{v-v_h}_{H_h^{-2}(G_0)}.
\end{equation}
Here the term proportional to \(h d^{-1}
\seminorm{\widetilde{\eta}^2e_h}_{H_h^2(G_0)}\) is absorbed into the left-hand
side by choosing \(\kappa\) sufficiently large.

It remains to bound the lower-order norms of \(e_h\).  By Assumption
\ref{assumption}(A3), the Poincar\'e inequality on the patch \(G\), and the
stability of \(P_h^G\),
\[
\begin{aligned}
    \sum_{j=0}^1 d^{j-2}
    \seminorm{P_h^G(\eta v)}_{H_h^j(G_0)}
    &\lesssim
    \seminorm{P_h^G(\eta v)}_{H_h^2(G)}
    \le
    \seminorm{\eta v}_{H_h^2(G)}
    \lesssim
    \sum_{j=0}^2 d^{j-2}\seminorm{v}_{H_h^j(G)}.
\end{aligned}
\]
Consequently,
\begin{equation}\label{lower order eh}
\begin{aligned}
    \sum_{j=0}^1 d^{j-2}\seminorm{e_h}_{H_h^j(G_0)}
    \lesssim
    \sum_{j=0}^2 d^{j-2}\seminorm{v}_{H_h^j(G)}
    +
    \sum_{j=0}^1 d^{j-2}\seminorm{v-v_h}_{H_h^j(G_0)}.
\end{aligned}
\end{equation}
Combining \eqref{first step local energy}, \eqref{sfinal improved}, and
\eqref{lower order eh} proves \eqref{local energy estimate}.
\end{proof}

\begin{remark}\upshape 
Taking \(v=u-\chi_h\) and \(v_h=u_h-\chi_h\), where \(u\) solves
\eqref{eq:continuous-weak-problem}, \(u_h\) solves the numerical scheme
\eqref{eq:discrete-scheme}, and \(\chi_h\in\mathring{V}_h\), Theorem
\ref{theorem local energy estimate} gives the standard local error estimate
\[
\begin{aligned}
       \seminorm{u-u_h}_{H_h^2(U)}
       \lesssim\;&
       \inf_{\chi_h\in\mathring{V}_h}
       \sum_{j=0}^2 d^{j-2}
       \seminorm{u-\chi_h}_{H_h^j(G)}
       +d^{-1}\seminorm{u-u_h}_{H_h^1(G_0)}
       \\
       &\quad
       +d^{-2}\norm{u-u_h}_{L^2(G_0)}
       +\seminorm{u-u_h}_{H_h^{-2}(G_0)}.
\end{aligned}
\]
The first term is the local approximation error, the next two terms are
lower-order local errors, and the last term is the residual measured in the
negative seminorm \eqref{eq:local-negative-seminorm}.  For conforming methods,
this residual vanishes when the test functions are supported in the patch.  For
nonconforming methods, it represents the local consistency error of the broken
scheme \eqref{eq:discrete-scheme}.  The main difficulty is the local
\(L^2\)-term, whose treatment is addressed in the following corollary.  For simplicity, we present the proof only for triangular elements; rectangular elements can be treated analogously.
\end{remark}

\begin{corollary}[Elimination of the local \(L^2\)-term]\label{corollary drop L2}
Let the assumptions of Theorem \ref{theorem local energy estimate} hold.
Suppose, in addition, that \(G_0\) satisfies one of the following two
conditions:
\[
    \text{(i) } G_0\cap \partial\Omega=\varnothing,
    \qquad
    \text{or}
    \qquad
    \text{(ii) } \operatorname{meas}(\partial G_0\cap\partial\Omega)\simeq d.
\]
Then, for any \(k\ge 2\),
\begin{align}\label{drop L2term}
\begin{aligned}
    \seminorm{u-u_h}_{H_h^2(U)}
    \lesssim\;&
    \inf_{\chi_h\in\mathring{V}_h}
    \sum_{j=0}^2 d^{j-2}
    \seminorm{u-\chi_h}_{H_h^j(G)}
    \\
    &+
    \inf_{\widetilde{\chi}_h\in P_k(G_0,\mathcal{T}_h)\cap H^1(G_0)}
    \sum_{j=0}^1 h^{j-1}d^{-1}
    \seminorm{u-\widetilde{\chi}_h}_{H^j(G_0)}
    \\
    &+
    d^{-1}\seminorm{u-u_h}_{H_h^1(G_0)}
    +
    \seminorm{u-u_h}_{H_h^{-2}(G_0)}.
\end{aligned}
\end{align}
\end{corollary}

\begin{proof}
We write \(e:=u-u_h\) and first consider the interior case \(G_0\cap\partial\Omega=\varnothing\). 
By taking \(v=u-\chi_h\) and \(v_h=u_h-\chi_h+P_{G_0}^0 e\) in Theorem \ref{theorem local energy estimate} ($v_h$ belongs to $V_h$), we have \(v-v_h=e-P_{G_0}^0e\) having zero mean over \(G_0\). Note that \(\seminorm{u-u_h}_{H_h^2(U)}=\seminorm{v-v_h}_{H_h^2(U)}\) and 
\(\seminorm{u-u_h}_{H_h^{-2}(U)}=\seminorm{v-v_h}_{H_h^{-2}(U)}\) in view of the definitions in \eqref{eq:broken-Hhm-seminorm} and \eqref{eq:local-negative-seminorm}. Therefore, the estimate \eqref{local energy estimate} in Theorem \ref{theorem local energy estimate} can be written as 
\begin{align}
\notag
        \seminorm{u-u_h}_{H_h^2(U)}
        \lesssim\;&
        \sum_{j=0}^2 d^{j-2}\seminorm{u-\chi_h}_{H_h^j(G)}
        + d^{-1}\seminorm{u-u_h}_{H_h^1(G_0)}
        + d^{-2}\norm{v-v_h}_{L^2(G_0)}
        \\ \label{local energy estimate u-uh}
        &\quad
        + \seminorm{u-u_h}_{H_h^{-2}(G_0)}.
\end{align}
It remains to show how the term \(d^{-2}\norm{v-v_h}_{L^2(G_0)}\) can be bounded by the right-hand side of \eqref{drop L2term}. To this end, we let \(Q_h^k\) be the \(L^2\)-projection onto
\(P_k(G_0,\mathcal{T}_h)\cap H^1(G_0)\), i.e., 
\begin{align}\label{def-Qhk}
    (Q_h^k v,w_h)_{G_0}
    =
    (v,w_h)_{G_0}
    \qquad
    \forall w_h\in P_k(G_0,\mathcal{T}_h)\cap H^1(G_0) ,
\end{align}
and use the triangle inequality 
\[
    \norm{v-v_h}_{L^2(G_0)}
    \le
    \norm{u-Q_h^ku }_{L^2(G_0)}
    +
    \norm{Q_h^ku-u_h-P_{G_0}^0e}_{L^2(G_0)}.
\]

Setting \(\widetilde e_h:=Q_h^k u-u_h-P_{G_0}^0e\) and choosing $w_h=1$ in \eqref{def-Qhk} (the latter gives $P_{G_0}^0Q_h^k v=P_{G_0}^0v$) yields that \(P_{G_0}^0\widetilde e_h=0\). Using the projection \(Q_h^1\), the
Poincar\'e inequality for zero-mean functions, inverse inequalities and $h\lesssim d$, we
obtain
\begin{align*}
    \norm{\widetilde e_h}_{L^2(G_0)}
    &\lesssim
    \norm{\widetilde e_h-Q_h^1\widetilde e_h}_{L^2(G_0)}
    +
    \norm{Q_h^1\widetilde e_h}_{L^2(G_0)}
    \\
    &\lesssim
    h(h+d)\seminorm{\widetilde e_h}_{H_h^2(G_0)}
    +
    d\seminorm{\widetilde e_h}_{H_h^1(G_0)}
    \\
    &\lesssim
    d\Big(
    \seminorm{e}_{H_h^1(G_0)}
    +
    \inf_{\widetilde{\chi}_h\in P_k(G_0,\mathcal{T}_h)\cap H^1(G_0)}
    \sum_{j=0}^1 h^{j-1}
    \seminorm{u-\widetilde{\chi}_h}_{H^j(G_0)}
    \Big).
\end{align*}
Therefore,
\[
    d^{-2}\norm{v-v_h}_{L^2(G_0)}
    \lesssim
    d^{-1}\seminorm{e}_{H_h^1(G_0)}
    +
    d^{-1}
    \inf_{\widetilde{\chi}_h\in P_k(G_0,\mathcal{T}_h)\cap H^1(G_0)}
    \sum_{j=0}^1 h^{j-1}
    \seminorm{u-\widetilde{\chi}_h}_{H^j(G_0)}.
\]
Substituting this result into \eqref{local energy estimate u-uh} yields the desired bound in case (i).

For boundary patches satisfying
\(\operatorname{meas}(\partial G_0\cap\partial\Omega)\simeq d\), take 
$v=u-\chi_h$ and $v_h=u_h-\chi_h.$
Replace $Q_h^k$ and $Q_h^1$ above with $I_h^k$ and $I_h^1$. Since $I_h^1(u-u_h)$ vanishes on the boundary portion, the same estimates follow, in which we apply a Poincar\'e inequality with a homogeneous boundary condition imposed on a portion of the boundary. Substituting this bound into Theorem \ref{theorem local energy estimate}
proves \eqref{drop L2term}.
\end{proof}

\begin{remark}
The elimination of the local \(L^2\)-term is one of the key points of the
analysis.  Estimating this term by a direct duality argument would require
\(H^4\)-regularity of the dual biharmonic problem, which is generally not
available on polygonal domains.  Estimating it by a global lower-order bound
would introduce logarithmic factors.  Corollary \ref{corollary drop L2} avoids
both difficulties and is therefore essential for the sharp pointwise estimates
proved later.
\end{remark}

The next subsection establishes the H\"older estimates for the Green's
function that allow the local bounds above to be summed over dyadic annuli
without producing logarithmic losses.

\subsection{H{\"o}lder estimates for the Green's function}
\label{subsec:holder-green}

The local estimate in Corollary~\ref{corollary drop L2} reduces the
pointwise analysis to estimating finite element errors of regularized Green's
functions on dyadic annuli.  In order to sum the annular contributions without
logarithmic losses, one needs a positive H\"older gain away from the pole of
the Green's function.  This subsection provides precisely this ingredient.

Let \(G(x,\zeta)\) denote the Green's function associated with the clamped
biharmonic operator in \eqref{definition model problem}, or equivalently with
the bilinear form \(a(\cdot,\cdot)\) in
\eqref{eq:continuous-bilinear-form}.  The sign convention for \(G\) is
irrelevant for the estimates below, since only absolute values of derivatives
are used.  The estimates proved here will be used later in two places:
\eqref{holder estimate2} is used to estimate the local approximation term for
the regularized second-derivative Green's function in \eqref{C2sigma}, while
\eqref{holder estimate1} is used in the duality estimate \eqref{psiC2sigma}.

\begin{lemma}[H\"older estimates for derivatives of the Green's function]
\label{lemma holder estimates}
Let \(\Omega\) be a convex polygon.  There exists
\(\sigma\in(0,1)\), depending only on \(\Omega\), such that for all
\(x,y,\zeta\in\Omega\),
\begin{equation}\label{holder estimate1}
    \frac{
    |D_x^2D_{\zeta}G(x,\zeta)
      -D_y^2D_{\zeta}G(y,\zeta)|
    }{|x-y|^{\sigma}}
    \lesssim
    |x-\zeta|^{-1-\sigma}
    +
    |y-\zeta|^{-1-\sigma},
\end{equation}
and
\begin{equation}\label{holder estimate2}
    \frac{
    |D_x^2D_{\zeta}^2G(x,\zeta)
      -D_y^2D_{\zeta}^2G(y,\zeta)|
    }{|x-y|^{\sigma}}
    \lesssim
    |x-\zeta|^{-2-\sigma}
    +
    |y-\zeta|^{-2-\sigma}.
\end{equation}
\end{lemma}

\begin{proof}
We prove \eqref{holder estimate2}; the proof of
\eqref{holder estimate1} is identical, with one fewer derivative in the
\(\zeta\)-variable.  Let \(\omega\) be a vertex of \(\Omega\), and set $\rho_\omega(x):=|x-\omega|.$
For simplicity, we consider only the corner singularity at \(\omega\).
The pointwise Green's function estimates of Maz'ya and Rossmann
\cite[p.~286]{MR1143406} imply that, for any multi-indices
\(\alpha,\beta\),
\begin{subequations}\label{eq:Mazya-Rossmann-pointwise}
\begin{align}
|D_x^\alpha D_{\zeta}^\beta G(x,\zeta)|
&\leq
C|x-\zeta|^{2-|\alpha|-|\beta|},
\quad
\begin{gathered}
\rho_\omega(x)/4\leq \rho_\omega(\zeta)\leq 4\rho_\omega(x),\\
3\leq |\alpha|+|\beta|,
\end{gathered}
\label{pointwise1}
\\
|D_x^\alpha D_{\zeta}^\beta G(x,\zeta)|
&\leq
C\rho_\omega(x)^{\gamma-|\alpha|}
 \rho_\omega(\zeta)^{-\gamma+2-|\beta|},
\quad
\begin{gathered}
\rho_\omega(x)\leq \frac34\rho_\omega(\zeta),\\
3\leq |\alpha|+|\beta|,
\end{gathered}
\label{pointwise2}
\\
|D_x^\alpha D_{\zeta}^\beta G(x,\zeta)|
&\leq
C\rho_\omega(x)^{-\gamma+2-|\alpha|}
 \rho_\omega(\zeta)^{\gamma-|\beta|},
\quad
\begin{gathered}
\rho_\omega(\zeta)\leq \frac34\rho_\omega(x),\\
3\leq |\alpha|+|\beta|.
\end{gathered}
\label{pointwise3}
\end{align}
\end{subequations}
Here \(\gamma\in(1-c_\Omega,1+c_\Omega)\), where
\(c_\Omega>1\) depends only on \(\Omega\); see \cite[Lemma~1.2]{MR614817}.  We choose
$0<\sigma<c_\Omega-1,$
so that \(\gamma=2+\sigma\) is admissible in \eqref{pointwise2}.  We shall
also use the local singular estimate
\begin{equation}\label{eq:local-green-singular-bound}
    |D_x^\alpha D_{\zeta}^\beta G(x,\zeta)|
    \lesssim
    |x-\zeta|^{2-|\alpha|-|\beta|},
    \qquad
    1\leq |\alpha|,|\beta|\leq 2,\quad
    3\leq |\alpha|+|\beta|\leq 4,
\end{equation}
which follows, for example, from \cite[Lemma~2.1]{MR4166452}.

Set
\[
    r_x:=|x-\zeta|,
    \qquad
    r_y:=|y-\zeta|,
    \qquad
    s:=|x-y|.
\]
We consider two cases.

\medskip
\noindent{\bf Case 1: \(r_x\leq 6s\).}
Then \(r_y\leq r_x+s\leq 7s\).  Using
\eqref{eq:local-green-singular-bound} with
\(|\alpha|=|\beta|=2\), we get
\[
\begin{aligned}
    \frac{
    |D_x^2D_{\zeta}^2G(x,\zeta)
      -D_y^2D_{\zeta}^2G(y,\zeta)|
    }{s^\sigma}
    &
    \leq
    s^{-\sigma}
    \bigl(
    |D_x^2D_{\zeta}^2G(x,\zeta)|
    +
    |D_y^2D_{\zeta}^2G(y,\zeta)|
    \bigr)
    \\
    &
    \lesssim
    r_x^{-2-\sigma}+r_y^{-2-\sigma}.
\end{aligned}
\]
This proves \eqref{holder estimate2} in Case 1.

\medskip
\noindent{\bf Case 2: \(r_x>6s\).}
For every point \(z_t:=tx+(1-t)y\), \(0\leq t\leq 1\), convexity of
\(\Omega\) gives \(z_t\in\Omega\), and the triangle inequality gives
\begin{equation}\label{eq:segment-distance-from-pole}
    |z_t-\zeta|
    \geq
    |x-\zeta|-|x-z_t|
    \geq
    \frac56 |x-\zeta|
    =
    \frac56 r_x.
\end{equation}
Moreover, \(r_y\simeq r_x\).  By the mean value theorem,
\begin{equation}\label{eq:mvt-holder-green}
\frac{
|D_x^2D_{\zeta}^2G(x,\zeta)
  -D_y^2D_{\zeta}^2G(y,\zeta)|
}{s^\sigma}
\leq
|D_{z_t}^3D_{\zeta}^2G(z_t,\zeta)|\,s^{1-\sigma}
\end{equation}
for some \(t\in[0,1]\).  It remains to bound the right-hand side of
\eqref{eq:mvt-holder-green}.  We distinguish three subcases according to the
relative sizes of \(\rho_\omega(x)\) and \(\rho_\omega(\zeta)\).

\medskip
\noindent{\bf Subcase 2(a):}
$\frac12\rho_\omega(x)\leq \rho_\omega(\zeta)
    \leq 2\rho_\omega(x).$
Using \eqref{eq:segment-distance-from-pole}, the above comparability implies
\[
    \frac14\rho_\omega(\zeta)
    \leq
    \rho_\omega(z_t)
    \leq
    \frac{5}{2}\rho_\omega(\zeta).
\]
Therefore \eqref{pointwise1}, with \(|\alpha|=3\) and
\(|\beta|=2\), yields    $|D_{z_t}^3D_{\zeta}^2G(z_t,\zeta)|
    \lesssim
    |z_t-\zeta|^{-3}$.
Combining this estimate with \eqref{eq:segment-distance-from-pole} and
\(s<r_x/6\), we obtain
\[
    |D_{z_t}^3D_{\zeta}^2G(z_t,\zeta)|\,s^{1-\sigma}
    \lesssim
    r_x^{-3}r_x^{1-\sigma}
    =
    r_x^{-2-\sigma}.
\]

\medskip
\noindent{\bf Subcase 2(b):}
$2\rho_\omega(x)<\rho_\omega(\zeta).$
Then
\begin{equation}\label{eq:rho-zeta-comparable-to-rx}
    \frac12\rho_\omega(\zeta)
    \leq
    r_x
    \leq
    \frac32\rho_\omega(\zeta)
\end{equation}
and $\rho_\omega(y)\leq \rho_\omega(x)+s\leq \frac34 \rho_\omega(\zeta)$. 
If \(s\geq \frac12\rho_\omega(x)\), then
\(\rho_\omega(y)\leq \rho_\omega(x)+s\lesssim s\).  Hence, using
\eqref{pointwise2} with \(|\alpha|=|\beta|=2\) and
\(\gamma=2+\sigma\),
\[
\begin{aligned}
    &\frac{
    |D_x^2D_{\zeta}^2G(x,\zeta)
      -D_y^2D_{\zeta}^2G(y,\zeta)|
    }{s^\sigma}
    \lesssim
    \left[
    \left(\frac{\rho_\omega(x)}{s}\right)^\sigma
    +
    \left(\frac{\rho_\omega(y)}{s}\right)^\sigma
    \right]
    \rho_\omega(\zeta)^{-2-\sigma}
    \lesssim
    r_x^{-2-\sigma},
\end{aligned}
\]
where the last step uses \eqref{eq:rho-zeta-comparable-to-rx}.

If  \(s<\frac12\rho_\omega(x)\), then
\(\rho_\omega(z_t)\simeq\rho_\omega(x)\) and
\(\rho_\omega(z_t)\leq \rho_\omega(\zeta)\).  Applying
\eqref{pointwise2} with \(|\alpha|=3\), \(|\beta|=2\), and
\(\gamma=2+\sigma\), we obtain
\[
\begin{aligned}
    |D_{z_t}^3D_{\zeta}^2G(z_t,\zeta)|\,s^{1-\sigma}
    &\lesssim
    \rho_\omega(z_t)^{\sigma-1}
    \rho_\omega(\zeta)^{-2-\sigma}
    s^{1-\sigma}
\lesssim
    \rho_\omega(\zeta)^{-2-\sigma}
    \lesssim
    r_x^{-2-\sigma}.
\end{aligned}
\]

\medskip
\noindent{\bf Subcase 2(c):}
$2\rho_\omega(\zeta)<\rho_\omega(x).$
Then \(\frac12\rho_\omega(x)\leq r_x\leq \frac32\rho_\omega(x)\).  Since \(s<r_x/6\), we have
\(\rho_\omega(z_t)\geq \rho_\omega(x)-s \geq \frac34 \rho_\omega(x)\), and hence
\(\rho_\omega(\zeta)\leq \rho_\omega(z_t)\).  Applying
\eqref{pointwise3} with \(|\alpha|=3\), \(|\beta|=2\), and
\(\gamma=2\), we get
\[
\begin{aligned}
    |D_{z_t}^3D_{\zeta}^2G(z_t,\zeta)|\,s^{1-\sigma}
    &\lesssim
    \rho_\omega(z_t)^{-3}s^{1-\sigma}
    \lesssim
    r_x^{-3}r_x^{1-\sigma}
    =
    r_x^{-2-\sigma}.
\end{aligned}
\]
This completes the proof of \eqref{holder estimate2}.  Since
\(r_y\simeq r_x\) in Case 2, the right-hand side is bounded by
$
    |x-\zeta|^{-2-\sigma}
    +
    |y-\zeta|^{-2-\sigma}.
$

Finally, the proof of \eqref{holder estimate1} follows from the same argument,
with \(|\beta|=1\) in place of \(|\beta|=2\).  The powers of
\(|x-\zeta|\) and \(|y-\zeta|\) are then reduced by one, giving
\eqref{holder estimate1}.
\end{proof}

\subsection{Estimates for regularized Green's functions}
\label{subsec:regularized-green}

The H\"older estimates of Lemma~\ref{lemma holder estimates} will be combined
with the local energy estimate \eqref{drop L2term} in the dyadic argument of
Section~\ref{sec:pointwise-estimates}.  Before doing so, we need global \(W^{m,1}\)-bounds for the
regularized Green's functions.  These estimates quantify the size of the
regularized singularities at the mesh scale \(h\), and they are used later in
the proofs of Theorems~\ref{theorem conforming main results} and
\ref{theorem nonconforming main results}.

Fix \(z\in\overline{\Omega}\), and let \(T_z\in\mathcal{T}_h\) be such that
\(z\in\overline{T_z}\).  There exists a regularized delta function
\(\delta=\delta_z\in C_0^\infty(T_z)\) such that
\begin{subequations}\label{eq:regularized-delta-properties}
\begin{align}
    (v,\delta_z)_{T_z}
    &=v(z)
    \qquad
    \forall v\in P_k(T_z),
    \label{eq:delta-reproducing-property}
    \\    \norm{D^m\delta_z}_{L^p(T_z)}
    &\lesssim
    h^{-m-2(1-1/p)},
    \qquad 1\leq p\leq\infty,\quad m=0,1,2,\ldots.
    \label{eq:delta-Lp-scaling}
\end{align}
\end{subequations}
The integer \(k\) in the reproducing property is chosen large enough so that
all elementwise polynomials whose point values are used below are reproduced.
In particular, if the local finite element degree is $\tilde{k}$, then we take
$k=\tilde{k}$.  Point values of derivatives on interelement edges are always
understood elementwise; this is sufficient because all \(W^{m,\infty}\)-norms
below are broken norms.

Here and below the hidden constants are independent of \(z\) and \(h\).  The
first property is the usual polynomial reproducing property of the
regularized delta function; see, for example, \cite{MR1388486}.  The scaling
bounds \eqref{eq:delta-Lp-scaling} will be
used repeatedly below.

For a multi-index \(\alpha\) with \(|\alpha|=i\), define
\(g_\alpha=g_\alpha^z\in H_0^2(\Omega)\) by
\begin{equation}\label{definition of regularized Green's functions}
    a(g_\alpha,v)
    =
    (D^\alpha v,\delta_z)
    =
    (-1)^{|\alpha|}(D^\alpha\delta_z,v)
    \qquad \forall v\in H_0^2(\Omega).
\end{equation}
Equivalently, in the distributional sense,
\begin{equation}\label{eq:correct-green-strong-form}
    \Delta^2 g_\alpha
    =
    (-1)^{|\alpha|}D^\alpha\delta_z
    \quad\text{in }\Omega,\qquad
    g_\alpha=\frac{\partial g_\alpha}{\partial n}=0
    \quad\text{on }\partial\Omega .
\end{equation}
When only the order of the derivative is relevant, we write \(g_i\) instead of
\(g_\alpha\).

\begin{lemma}[A priori estimates for regularized Green's functions]
\label{lemma a priori estimates for gi}
Let \(\Omega\) be a convex polygon.  The regularized Green's functions defined
by \eqref{definition of regularized Green's functions} satisfy
\begin{subequations}\label{eq:regularized-green-apriori-estimates}
\begin{align}
\label{g1 W41}
    \seminorm{g_1}_{W^{4,1}(\Omega)}
    &\lesssim h^{-1},
\\
\label{g0 W41}
    \seminorm{g_0}_{W^{4,1}(\Omega)}
    &\lesssim |\ln h|,
\\
\label{g2 W31}
    \seminorm{g_2}_{W^{3,1}(\Omega)}
    &\lesssim h^{-1},
\\
\label{g1 W31}
    \seminorm{g_1}_{W^{3,1}(\Omega)}
    &\lesssim |\ln h|.
\end{align}
\end{subequations}
\end{lemma}

\begin{proof}
We give the details for \eqref{g1 W41}; the other estimates follow from the
same decomposition, and the changes are summarized at the end of the proof.

Let \(\mathcal{V}\) be the set of vertices of \(\Omega\).  Choose $
    d_0:=\kappa h,\ 
    d_j:=2^j d_0,\ j=0,1,\ldots,J,$
where \(\kappa\) is the constant used in the local estimates and \(J\) is
chosen such that
$\operatorname{diam}(\Omega)\leq d_J<2\operatorname{diam}(\Omega).$
Define the near-singular region
\begin{equation}\label{eq:near-singular-region}
    \Omega_0
    :=
    \Omega\cap \big(B(z,d_0)
    \cup
    \bigcup_{\omega\in\mathcal{V}} B(\omega,d_0)\big),
\end{equation}
and, for \(j\geq1\), the dyadic annuli centered at \(z\),
\begin{equation}\label{eq:dyadic-annuli-z}
    A_j(z)
    :=
    \{x\in\Omega:\ d_{j-1}<|x-z|\leq d_j\}.
\end{equation}
Since \(\Omega\) has only finitely many vertices, all estimates below are
uniform after summing over \(\omega\in\mathcal{V}\).

We split
\begin{equation}\label{eq:g1-W41-near-far-split}
    \seminorm{g_1}_{W^{4,1}(\Omega)}
    \leq
    \seminorm{g_1}_{W^{4,1}(\Omega_0)}
    +
    \seminorm{g_1}_{W^{4,1}(\Omega\setminus\Omega_0)}
    =:B_1+B_2.
\end{equation}

For the near-field term \(B_1\), we use the regularity estimate for the
clamped biharmonic problem on convex polygons (see Appendix~\ref{supp:regularity} for a detailed proof): there exists \(p_0>1\),
depending only on \(\Omega\), such that
\begin{equation}\label{W4p regulatity}
    \seminorm{w}_{W^{4,p}(\Omega)}
    \lesssim
    \norm{F}_{L^p(\Omega)}
    \qquad
    \forall\,1<p<p_0,
\end{equation}
whenever \(w\) solves the clamped biharmonic problem with right-hand side
\(F\).  Fix \(1<p<p_0\).  Since
\(|\Omega_0|\lesssim h^2\), \eqref{definition of regularized Green's functions},
\eqref{W4p regulatity}, and \eqref{eq:delta-Lp-scaling} imply
\begin{equation}\label{eq:g1-W41-near-estimate}
\begin{aligned}
    B_1
    &\lesssim
    |\Omega_0|^{1-1/p}
    \seminorm{g_1}_{W^{4,p}(\Omega)}
    \lesssim
    h^{2-2/p}
    \norm{D\delta_z}_{L^p(T_z)}
    \lesssim
    h^{2-2/p} h^{-3+2/p}
    \lesssim h^{-1}.
\end{aligned}
\end{equation}

We next estimate the far-field term \(B_2\).  Using the representation of
\(g_1\) by the Green's function, integrating by parts in the \(\zeta\)-variable
so that the derivative falls on \(G\), and using
\eqref{eq:delta-Lp-scaling}, we obtain from
\eqref{pointwise1}--\eqref{pointwise3} the pointwise bound
\begin{equation}\label{eq:g1-far-kernel-bound}
    |D^4 g_1(x)|
    \lesssim
    |x-z|^{-3}
    +
    \sum_{\omega\in\mathcal{V}}
    |x-\omega|^{-3},
    \qquad x\in\Omega\setminus\Omega_0.
\end{equation}
Consequently,
\begin{equation}\label{eq:g1-W41-far-estimate}
\begin{aligned}
    B_2
    \lesssim
    \int_{\Omega\setminus\Omega_0}
    \left(
    |x-z|^{-3}
    +
    \sum_{\omega\in\mathcal{V}} |x-\omega|^{-3}
    \right)\,\mathrm{d}x
    \lesssim
    \int_{\kappa h}^{\operatorname{diam}(\Omega)}
    r^{-3}\,r\,\mathrm{d}r
    \lesssim h^{-1}.
\end{aligned}
\end{equation}
Combining \eqref{eq:g1-W41-near-far-split},
\eqref{eq:g1-W41-near-estimate}, and
\eqref{eq:g1-W41-far-estimate} proves \eqref{g1 W41}.

It remains to indicate the corresponding estimates for \(g_0\), \(g_2\), and
the \(W^{3,1}\)-seminorm of \(g_1\).  We use, in addition to
\eqref{W4p regulatity}, the lower-order consequence (see \cite[Corollary 7.3.2.5]{MR775683})
\begin{equation}\label{eq:W3p-negative-regularity}
    \seminorm{w}_{H^3(\Omega)}
    \lesssim
    \norm{F}_{H^{-1}(\Omega)}
\end{equation}
for the same clamped biharmonic problem.

For the near-field part, \eqref{W4p regulatity},
\eqref{eq:W3p-negative-regularity}, and
\eqref{eq:delta-Lp-scaling} give
\begin{subequations}\label{eq:near-field-other-gi}
\begin{align}
    \seminorm{g_0}_{W^{4,1}(\Omega_0)}
    &\lesssim
    h^{2-2/p}\norm{\delta_z}_{L^p(T_z)}
    \lesssim 1,
    \label{eq:g0-W41-near}
    \\
    \seminorm{g_2}_{W^{3,1}(\Omega_0)}
    &\lesssim
    h
    \norm{D^2\delta_z}_{H^{-1}(\Omega)}
    \lesssim
    h\norm{D\delta_z}_{L^2(T_z)}
    \lesssim h^{-1},
    \label{eq:g2-W31-near}
    \\
    \seminorm{g_1}_{W^{3,1}(\Omega_0)}
    &\lesssim
    h
    \norm{D\delta_z}_{H^{-1}(\Omega)}
    \lesssim
    h\norm{\delta_z}_{L^2(T_z)}
    \lesssim 1.
    \label{eq:g1-W31-near}
\end{align}
\end{subequations}

For the far-field part, the same Green's function representation and
\eqref{pointwise1}--\eqref{pointwise3} yield
\begin{equation}\label{eq:far-field-kernel-orders}
\begin{array}{c|c|c}
\text{quantity}
&
\text{far-field pointwise bound}
&
\text{integral over }\Omega\setminus\Omega_0
\\ \hline
|D^4g_0(x)|
&
|x-z|^{-2}+\displaystyle\sum_{\omega\in\mathcal{V}}|x-\omega|^{-2}
&
O(|\ln h|)
\\[2mm]
|D^3g_2(x)|
&
|x-z|^{-3}+\displaystyle\sum_{\omega\in\mathcal{V}}|x-\omega|^{-3}
&
O(h^{-1})
\\[2mm]
|D^3g_1(x)|
&
|x-z|^{-2}+\displaystyle\sum_{\omega\in\mathcal{V}}|x-\omega|^{-2}
&
O(|\ln h|).
\end{array}
\end{equation}
Indeed,
\[
    \int_{\kappa h}^{\operatorname{diam}(\Omega)} r^{-2}r\,\mathrm{d}r
    \lesssim |\ln h|,
    \qquad
    \int_{\kappa h}^{\operatorname{diam}(\Omega)} r^{-3}r\,\mathrm{d}r
    \lesssim h^{-1}.
\]
Combining the near-field bounds \eqref{eq:near-field-other-gi} with the
far-field bounds \eqref{eq:far-field-kernel-orders} proves
\eqref{g0 W41}, \eqref{g2 W31}, and \eqref{g1 W31}.  The proof is complete.
\end{proof}

\section{Pointwise error estimates}
\label{sec:pointwise-estimates}

We now combine the local error estimate \eqref{drop L2term} with the H\"older
estimates \eqref{holder estimate1} and \eqref{holder estimate2}.  We also use
the regularized Green's function bounds \eqref{g1 W41}--\eqref{g1 W31}.
The main technical step is a uniform
\(W^{2,1}\)-estimate for the finite element error of the regularized
second-derivative Green's function.

For \(z\in\overline{\Omega}\) and $ |\alpha|=2$, let \(g_\alpha =g_\alpha^z\) be defined by
\eqref{definition of regularized Green's functions} and let
\(g_{\alpha,h}=g_{\alpha,h}^z\in\mathring{V}_h\) be its finite element approximation
defined by 
\begin{equation}\label{eq:correct-discrete-green}
    a_h(g_{\alpha,h},v_h)
    =
    (-1)^{|\alpha|}(D^\alpha\delta_z,v_h)
    \qquad \forall v_h\in\mathring V_h .
\end{equation}

Before giving the finite element error of $g_{\alpha}-g_{\alpha,h}$, we shall state the consistency error estimate for the Morley element; see \cite{Z.Shi1}, which will be used repeatedly in the subsequent analysis. As the result will be used later in slightly different forms, we include a detailed proof for completeness, which may be found in Appendix~\ref{supp:Morley}.

\begin{lemma}[Morley consistency estimate]\label{lemma Morley consistency estimate}
Let $V_h$ denote the Morley element space. Let $u$ solve \eqref{eq:continuous-weak-problem}. If $f\in L^{p}(\Omega)$  and $u\in W^{3,p}(\Omega)$, for any $w\in V+\mathring{V}_h$ and $D\subset\Omega$, it holds that
\begin{align}
\begin{aligned}\label{split-version Morley consistency estimate}
|a_h(u,w)-(f,w)|
&\lesssim  |(f,w-I_h^1w)|+h\seminorm{u}_{W^{3,p}(D\cap \Omega_w)}\seminorm{w}_{W_h^{2,p'}(D\cap \Omega_w)}\\
&\ \ \ \ +h^{\sigma}\seminorm{u}_{C^{2,\sigma}(\Omega_w\backslash D)}\seminorm{w}_{W^{2,1}_h(\Omega_w\backslash D)},   
\end{aligned}
\end{align}
where $1\leq p\leq \infty$, $\Omega_w:= \operatorname{supp }w$ and \(\sigma\in(0,1)\) is the exponent from
Lemma~\ref{lemma holder estimates}.
\end{lemma} 

\begin{theorem}\label{theorem g-gh} Let \(V_h\) be either a conforming finite element space
or the Morley finite element space. Under Assumption~\ref{assumption}, the following holds uniformly with
respect to \(z\in\overline{\Omega}\) and \(h\):
\begin{equation}\label{g2-g2h}
    \seminorm{g_\alpha - g_{\alpha,h}}_{W_h^{2,1}(\Omega)}
    \lesssim 1 \quad\mbox{for}\,\,\, |\alpha|=2 .
\end{equation}

\end{theorem}

\begin{proof}
We give the proof for the Morley element. The conforming case is simpler,
because the local consistency terms vanish. 
For simplicity of notation, we write \(g_2\) for  
\(g_\alpha\) when the $\alpha$ satisfying $|\alpha|=2$ is fixed. 

Set \(e:=g_2-g_{2,h}\) and let \(d_0:=\kappa h\), \(d_j:=2^j d_0\) for \(j=0,1,\ldots,J\), where \(J\) is chosen such that \(\operatorname{diam}(\Omega)\le d_J<2\operatorname{diam}(\Omega)\). 
We use the dyadic annuli
\begin{equation}\label{eq:pointwise-annuli}
    A_0:=\Omega\cap B(z,d_0),\qquad
    A_j:=\{x\in\Omega:\ d_{j-1}<|x-z|\le d_j\},\quad j\ge 1 .
\end{equation}
For \(j\ge1\), define the enlarged annuli
\begin{equation}\label{eq:pointwise-enlarged-annuli}
    A_j^*
    :=
    \left\{
    x\in\Omega:\ \frac23 d_{j-1}<|x-z|\le \frac32 d_j
    \right\},
\end{equation}
and let \(A_j^{**}\) be a slightly larger element patch satisfying
\[
    A_j\subset A_j^*\subset A_j^{**},\qquad
    \operatorname{diam}(A_j^{**})\simeq d_j,\qquad
    \operatorname{dist}(\partial A_{j^*}\backslash \partial\Omega,\partial A_j^{**}\backslash \partial\Omega)\simeq d_j .
\]
Moreover, we define the elementwise enlargement of $A_j$,
\[
\tilde{A}_j
:=
\operatorname{int}
\left(
\bigcup_{\substack{T\in\mathcal{T}_h,\, T\cap A_j\neq\emptyset}}
\overline{T}
\right).
\]
For notational simplicity, we continue to write $A_j$ for
$\widetilde{A}_j$. Likewise, $A_j^*$ and $A_j^{**}$ are understood as
their corresponding elementwise enlargements. Note that since $h/d_j \leq 1/\kappa$ with a sufficiently large $\kappa$, this convention will not essentially affect $d_j$.

If an annulus intersects the boundary, we cover it by a uniformly finite
number of element patches satisfying one of the two alternatives in
Corollary~\ref{corollary drop L2}.  To simplify notation, we write the
argument directly on \(A_j,A_j^*,A_j^{**}\).

By Cauchy's inequality on each annulus,
\begin{equation}\label{eq:W21-dyadic-decomposition}
\begin{aligned}
    \seminorm{e}_{W_h^{2,1}(\Omega)}
    &\lesssim
    d_0\seminorm{e}_{H_h^2(A_0)}
    +
    \sum_{j=1}^J d_j\seminorm{e}_{H_h^2(A_j)} .
\end{aligned}
\end{equation}
We estimate the first term by 
the global energy estimate (see (10) in \cite{Z.Shi1})
\begin{equation*} \seminorm{e}_{H_h^2(\Omega)}\lesssim h(\seminorm{g_2}_{H^3(\Omega)}+h\norm{D^2\delta_z}_{L^2(T_z)})  
\end{equation*}
Using \eqref{eq:delta-Lp-scaling} and $H^3$-regularity estimate  $\seminorm{g_2}_{H^3(\Omega)}
    \lesssim
    \norm{D^2\delta_z}_{H^{-1}(\Omega)}$,
we obtain
\begin{equation}\label{eq:E1-estimate}
\begin{aligned}
    d_0\seminorm{e}_{H_h^2(A_0)}
    &\lesssim
    \kappa h^2
    \left(
    \seminorm{g_2}_{H^3(\Omega)}
    +
    h\norm{D^2\delta_z}_{L^2(T_z)}
    \right)
    \lesssim 1 .
\end{aligned}
\end{equation}

We next estimate the annular contributions in
\eqref{eq:W21-dyadic-decomposition}.  Applying
Corollary~\ref{corollary drop L2} with
$U=A_j,\ G_0=A_j^*,\ G=A_j^{**},$
and using local interpolation for \(g_2\), we get
\begin{equation}\label{eq:local-annular-estimate}
\begin{aligned}
    d_j\seminorm{e}_{H_h^2(A_j)}
    \lesssim\;&
    h^\sigma d_j^2
    \seminorm{g_2}_{C^{2,\sigma}(A_j^{**})}
    +
    \seminorm{e}_{H_h^1(A_j^*)}
    +
    d_j\seminorm{e}_{H_h^{-2}(A_j^*)},
\end{aligned}
\end{equation}
where \(\sigma\in(0,1)\) is the exponent from
Lemma~\ref{lemma holder estimates} and the derivation of the term
\(h^\sigma d_j^2\seminorm{g_2}_{C^{2,\sigma}(A_j^{**})}\) involves interpolation approximation and the fact that $h\le d_j$.

Summing \eqref{eq:local-annular-estimate} over \(j\ge1\), and using
\eqref{eq:E1-estimate}, gives
\begin{equation}\label{eq:E-decomposition}
\begin{aligned}
 d_0\seminorm{e}_{H_h^2(A_0)}
    +
    \sum_{j=1}^J d_j\seminorm{e}_{H_h^2(A_j)}
    \lesssim\;&
    1
    +
    \underbrace{
    \sum_{j=1}^J
    h^\sigma d_j^2
    \seminorm{g_2}_{C^{2,\sigma}(A_j^{**})}
    }_{=:E_2}
    \\
    &+
    \underbrace{
    \sum_{j=1}^J
    \seminorm{e}_{H_h^1(A_j^*)}
    }_{=:E_3}
    +
    \underbrace{
    \sum_{j=1}^J
    d_j\seminorm{e}_{H_h^{-2}(A_j^*)}
    }_{=:E_4}.
\end{aligned}
\end{equation}

We first estimate \(E_2\).  Since \(T_z\subset A_0\) and
\(A_j^{**}\cap T_z=\varnothing\) for \(j\ge1\), the representation of \(g_2\)
by the Green's function and \eqref{holder estimate2} yield
\begin{equation}\label{C2sigma}
\begin{aligned}
    \seminorm{g_2}_{C^{2,\sigma}(A_j^{**})}
    &\lesssim
    \sup_{\substack{x,y\in A_j^{**}, \zeta\in T_z}}
    \frac{
    |D_x^2D_\zeta^2G(x,\zeta)
      -D_y^2D_\zeta^2G(y,\zeta)|
    }{|x-y|^\sigma}
    \norm{\delta_z}_{L^1(T_z)}
    \\
    &\lesssim
    \sup_{\substack{x,y\in A_j^{**}, \zeta\in T_z}}
    \left(
    |x-\zeta|^{-2-\sigma}
    +
    |y-\zeta|^{-2-\sigma}
    \right)
    \lesssim d_j^{-2-\sigma}.
\end{aligned}
\end{equation}
Consequently,
\begin{equation}\label{eq:E2-estimate}
    E_2
    \lesssim
    \sum_{j=1}^J
    \left(\frac{h}{d_j}\right)^\sigma
    \lesssim 1 .
\end{equation}

We now estimate \(E_3\).  Fix \(j\ge1\).  It is enough to estimate one first
derivative, say \(\partial_1 e\), on \(A_j^*\).  Owing to the discontinuity of $g_{2,h}$, we make some modifications to the standard duality argument. For each $ T\subset A_j^*$, let $v_T\in C_0^{\infty}(T)$ with $\norm{v_T}_{L^2(T)}=1$ satisfy 
   $\norm{\partial_1 e}_{L^2(T)}\leq 2(\partial_1 e,v_T)$.
Set 
$v:=\sum_{T\subset A_j^*} \frac{ \norm{\partial_1 e}_{L^2(T)}}{ \norm{\partial_1 e}_{L^2(A_j^*)}}v_T.$
Note that $v\in C_0^{\infty}(A_j^*)$ satisfies
$ \norm{v}_{L^2(A_j^*)}=1$ and vanishes on each edge contained in $A_j^*$.
Hence,
\begin{equation}\label{partial1e}
   \norm{\partial_1 e}_{L^2(A_j^*)} \leq 2(\partial_1 e,v)=2\sum_{T\subset A_j^*} \int_T v\partial_1 e \, {\rm d}x= -2(e,\partial_1 v).
\end{equation} 
Let \(\psi\in H_0^2(\Omega)\) solve the dual problem: $\Delta^2\psi=\partial_1 v$ in $\Omega$ with
    $\psi=\frac{\partial\psi}{\partial n}=0$
    on $\partial\Omega$.
Then
$\norm{\psi}_{H^3(\Omega)}\lesssim \norm{\partial_1 v}_{H^{-1}(\Omega
)}\lesssim \norm{v}_{L^2(A_j^*)}=1 .$
Moreover, by the Green's function representation of \(\psi\) and
\eqref{holder estimate1},
\begin{equation}\label{psiC2sigma}
\begin{aligned}
    \seminorm{\psi}_{C^{2,\sigma}(\Omega\setminus A_j^{**})}
    &\lesssim
    d_j^{-1-\sigma}\norm{v}_{L^1(A_j^*)}
    \lesssim
    d_j^{-1-\sigma}|A_j^*|^{1/2}\norm{v}_{L^2(A_j^*)}
    \lesssim d_j^{-\sigma}.
\end{aligned}
\end{equation}

Using the dual problem and \eqref{partial1e}, adding and subtracting \(a_h(e,\psi)\), and then using
the discrete equation for \(g_{2,h}\), we obtain
\begin{equation}\label{eq:E3-duality-decomposition}
\begin{aligned}
    \norm{\partial_1 e}_{L^2(A_j^*)}   
    &\lesssim\underbrace{\left| (\Delta^2 \psi,e)-a_h(\psi,e) \right|}_{:=E_{31}^{(j)}} 
    +\underbrace{\left| a_h(e,\psi-\Pi_h\psi) \right|}_{:=E_{32}^{(j)}}
    \\
    &\quad +\underbrace{\left| (\Delta^2 g_2,\psi-\Pi_h \psi)-a_h(g_2,\psi-\Pi_h \psi) \right|}_{:=E_{33}^{(j)}}
\end{aligned}
\end{equation}

\par We first estimate $E_{31}^{(j)}$.
Using \eqref{split-version Morley consistency estimate} with $D=A_j^{**}$ and $p=2$, the
\(H^3\)-bound for \(\psi\) on \(A_j^{**}\), and
\eqref{psiC2sigma} outside \(A_j^{**}\), we obtain
\begin{equation}\label{eq:E31-estimate}
\begin{aligned}
E_{31}^{(j)}
&\lesssim h\norm{v}_{L^2(A_j^{**})}
\seminorm{e}_{H_h^2(A_j^{**})}
 +h\seminorm{\psi}_{H^3(A_j^{**})}
\seminorm{e}_{H_h^2(A_j^{**})}
\\
&\quad
 +h^{\sigma}\seminorm{\psi}_{C^{2,\sigma}(\Omega\backslash A_j^{**})}
\seminorm{e}_{W_h^{2,1}(\Omega\backslash A_j^{**})}
\\
   & \lesssim   h\seminorm{e}_{H_h^2(A_j^{**})}
    +
    \left(\frac{h}{d_j}\right)^\sigma
    \seminorm{e}_{W_h^{2,1}(\Omega)} .    
\end{aligned}
\end{equation}
For \(E_{32}^{(j)}\), we use the
\(H^3\)-bound for \(\psi\) on \(A_j^{**}\), and
\eqref{psiC2sigma} outside \(A_j^{**}\), to get
\begin{equation}\label{eq:E32-estimate}
\begin{aligned}
E_{32}^{(j)} &\lesssim  h\seminorm{\psi}_{H^3(A_j^{**})}\seminorm{e}_{H_h^2(A_j^{**})}+h^{\sigma}\seminorm{\psi}_{C^{2,\sigma}(\Omega\backslash A_j^{**})}\seminorm{e}_{W_h^{2,1}(\Omega\backslash A_j^{**})}  \\
   & \lesssim   h\seminorm{e}_{H_h^2(A_j^{**})}
    +
    \left(\frac{h}{d_j}\right)^\sigma
    \seminorm{e}_{W_h^{2,1}(\Omega)} .    
\end{aligned}   
\end{equation}
Finally, we estimate \(E_{33}^{(j)}\). Notice that $T_z\subset \Omega\backslash A_j^{**}$.  Choosing  $D=\Omega\backslash A_j^{**}$ and $p=1$ in \eqref{split-version Morley consistency estimate} and \eqref{g2 W31}, together with
\eqref{psiC2sigma}, gives
\begin{equation}\label{eq:E33-estimate}
\begin{aligned}
    E_{33}^{(j)}&\lesssim h^{2+\sigma}\norm{D^2\delta_z}_{L^1(T_z)}\seminorm{\psi}_{C^{2,\sigma}(T_z)}+ h^{1+\sigma}\seminorm{g_2}_{W^{3,1}(\Omega\backslash A_j^{**})}\seminorm{\psi}_{C^{2,\sigma}(\Omega\backslash A_j^{**})}\\
    &\quad + h^{1+\sigma}d_j\seminorm{g_2}_{C^{2,\sigma}(A_j^{**})}\seminorm{\psi}_{H^3(A_j^{**})} \\
    &\lesssim \left(\frac{h}{d_j}\right)^\sigma+\left(\frac{h}{d_j}\right)^{1+\sigma}\lesssim  \left(\frac{h}{d_j}\right)^\sigma.
\end{aligned}
\end{equation}
Combining \eqref{eq:E3-duality-decomposition}--\eqref{eq:E33-estimate}, and
summing over the two first derivatives, gives
\begin{equation}\label{eq:local-H1-error-annulus}
    \seminorm{e}_{H_h^1(A_j^*)}
    \lesssim
    h\seminorm{e}_{H_h^2(A_j^{**})}
    +
    \left(\frac{h}{d_j}\right)^\sigma
    \seminorm{e}_{W_h^{2,1}(\Omega)}
    +
    \left(\frac{h}{d_j}\right)^\sigma .
\end{equation}
Therefore,
\begin{equation}\label{eq:E3-estimate}
\begin{aligned}
    E_3
    &\lesssim
    \sum_{j=1}^J
    h\seminorm{e}_{H_h^2(A_j^{**})}
    +
    \left(
    \sum_{j=1}^J
    \left(\frac{h}{d_j}\right)^\sigma
    \right)
    \seminorm{e}_{W_h^{2,1}(\Omega)}
    +
    1 .
\end{aligned}
\end{equation}

It remains to estimate \(E_4\).  By the definition
\eqref{eq:local-negative-seminorm}, for any
\(w_h\in\mathring{V}_h(A_j^*)\) with
\(\seminorm{w_h}_{H_h^2(A_j^*)}=1\),
\[
    a_h(e,w_h)
    =
    a_h(g_2,w_h)-(D^{2}\delta_z,w_h).
\]
Since \(A_j^*\cap T_z=\varnothing\) for \(j\ge1\), \eqref{split-version Morley consistency estimate} with $D=\varnothing$ gives
\[
    |a_h(e,w_h)|
    \lesssim
    h^\sigma d_j
    \seminorm{g_2}_{C^{2,\sigma}(A_j^{**})}
    \lesssim
    h^\sigma d_j^{-1-\sigma}.
\]
Multiplying by \(d_j\) and taking the supremum over \(w_h\) yields
\begin{equation}\label{eq:E4-estimate}
    E_4
    \lesssim
    \sum_{j=1}^J
    \left(\frac{h}{d_j}\right)^\sigma
    \lesssim 1 .
\end{equation}

Combining \eqref{eq:E-decomposition}, \eqref{eq:E2-estimate},
\eqref{eq:E3-estimate}, and \eqref{eq:E4-estimate}, we arrive at
\begin{align}\label{eq:g2-final-absorption}
    &d_0\seminorm{e}_{H_h^2(A_0)}
    +
    \sum_{j=1}^J d_j\seminorm{e}_{H_h^2(A_j)} \notag\\
    &\lesssim 
    1
    +
    \sum_{j=1}^J
    h\seminorm{e}_{H_h^2(A_j^{**})}
    +
    \left(
    \sum_{j=1}^J
    \left(\frac{h}{d_j}\right)^\sigma
    \right)
    \seminorm{e}_{W_h^{2,1}(\Omega)} \\
    &\lesssim 
    1
    +
    \sum_{j=1}^J
    h\seminorm{e}_{H_h^2(A_j^{**})}
    +
    \left(
    \sum_{j=1}^J
    \left(\frac{h}{d_j}\right)^\sigma
    \right) \Big(d_0\seminorm{e}_{H_h^2(A_0)}
    +
    \sum_{j=1}^J d_j\seminorm{e}_{H_h^2(A_j)}\Big) . \notag
\end{align}
Since \(d_j\ge \kappa h\), it follows that 
\(
    h\seminorm{e}_{H_h^2(A_j^{**})}
    =
    \frac{h}{d_j}
    d_j\seminorm{e}_{H_h^2(A_j^{**})}
    \lesssim \kappa^{-1} d_j\seminorm{e}_{H_h^2(A_j^{**})} .
\)
The enlarged annuli \(A_j^{**}\) have uniformly finite overlap, and therefore
\[
    \sum_{j=1}^J
    h\seminorm{e}_{H_h^2(A_j^{**})}
    \lesssim
    \kappa^{-1} \Big(d_0\seminorm{e}_{H_h^2(A_0)}
    +
    \sum_{j=1}^J d_j\seminorm{e}_{H_h^2(A_j)}\Big).
\]
Also, 
\(
    \sum_{j=1}^J
    \left(\frac{h}{d_j}\right)^\sigma
    \lesssim \kappa^{-\sigma}.
\)
Choosing \(\kappa\) sufficiently large, the last two terms in
\eqref{eq:g2-final-absorption} are absorbed into the left-hand side.  This
proves \eqref{g2-g2h}.
\end{proof}

\begin{corollary}\label{corollary W2p}
Define $\Gamma_z(x)
    :=
    \left(|x-z|^2+\kappa^2 h^2\right)^{1/2}.$
Under the assumptions in Theorem~\ref{theorem g-gh}, it holds that
\begin{equation}\label{eq:weighted-g2-error}
    \sup_{z\in\overline{\Omega}} M_{\sigma,h}^z
    \lesssim h^{\sigma/2},
\end{equation}
where \(\sigma\) is the exponent in
Lemma~\ref{lemma holder estimates} and 
     $M_{\sigma,h}^z
     :=
     \left(
     \int_{\Omega}
     \Gamma_z^{2+\sigma}
     |D^2(g_2^z-g_{2,h}^z)|^2\,{\rm d}x
     \right)^{1/2}$.
\end{corollary}

\begin{proof}
We still use $e$ to denote $ g_2-g_{2,h}$.   On \(A_j\) ($j\geq 0$), we have
\(\Gamma_z(x)\simeq d_j\). 
For $j=0$, we have by a global energy estimate
\begin{equation*}
\int_{A_0}
    \Gamma_z^{2+\sigma}
    |D^2e|^2\,{\rm d}x  \lesssim h^{2+\sigma}\norm{D^2e}_{L^2(\Omega)}^2\lesssim h^{\sigma}.  
\end{equation*}
For $j\geq 1$, the squared version of the annular estimate
\eqref{eq:local-annular-estimate}, together with
\eqref{C2sigma}, \eqref{eq:local-H1-error-annulus}, and
\eqref{eq:E4-estimate}, gives
\begin{equation}\label{eq:weighted-annular-bound}
\begin{aligned}
    \int_{A_j}
    \Gamma_z^{2+\sigma}
    |D^2e|^2\,{\rm d}x
    &\lesssim
    h^\sigma
    \left(\frac{h}{d_j}\right)^\sigma
    +
    h^2d_j^\sigma
    \seminorm{e}_{H_h^2(A_j^{**})}^2
    +
    h^\sigma
    \left(\frac{h}{d_j}\right)^\sigma
    \seminorm{e}_{W_h^{2,1}(\Omega)}^2 .
\end{aligned}
\end{equation}
By Theorem~\ref{theorem g-gh}, the last term is bounded by
\(h^\sigma(h/d_j)^\sigma\).  Summing \eqref{eq:weighted-annular-bound} over
\(j\), using finite overlap, and absorbing the term
\[
    \sum_j h^2d_j^\sigma    \seminorm{e}_{H_h^2(A_j^{**})}^2
    \le
    C\kappa^{-2}
    \sum_{j} d_j^{2+\sigma}
    \seminorm{e}_{H_h^2(A_j^{**})}^2
\]
into the left-hand side by taking \(\kappa\) sufficiently large, we obtain $(M_{\sigma,h}^z)^2\lesssim h^\sigma$ .
This proves \eqref{eq:weighted-g2-error}.
\end{proof}

\begin{lemma}\label{lemma weighted D3}
We have the weighted-norm estimate
\begin{equation}\label{eq:weighted-D3-g2}
    \int_{\Omega}
    \Gamma_z^{2+\sigma}|D^3g_2|^2\,{\rm d}x
    \lesssim h^{\sigma-2},
\end{equation}
where \(\sigma\) is the exponent in
Lemma~\ref{lemma holder estimates}.
\end{lemma}

\begin{proof}
We split the proof into two parts: the near-singular estimate and the far-singular estimate.
The near-singular contribution is estimated
by $H^3$-regularity:
\[
    \int_{\Omega\cap B(z,4d_0)}
    \Gamma_z^{2+\sigma}|D^3g_2|^2\,{\rm d}x
    \lesssim
    h^{2+\sigma}\seminorm{g_2}_{H^3(\Omega)}^2
    \lesssim
    h^{2+\sigma}h^{-4}
    =
    h^{\sigma-2}.
\]

It remains to consider \(\Omega\setminus B(z,4d_0)\).  Let \(\omega\) be a vertex of
\(\Omega\), and set \(\rho_\omega(x)=|x-\omega|\). For simplicity, we only consider the corner singularity associated with a single vertex.   

\par The pointwise Green's function
estimates \eqref{pointwise1}--\eqref{pointwise3}, with
\(\gamma=2+\sigma\) in \eqref{pointwise2}, imply the following two bounds:
away from the vertices,
\[
    |D^3g_2(x)|\lesssim |x-z|^{-3},
\]
while in the region where \(\rho_\omega(x)\le |z-\omega|/2\),
\[
    |D^3g_2(x)|
    \lesssim
    \rho_\omega(x)^{\sigma-1}
    |z-\omega|^{-2-\sigma}.
\]
Therefore,
\[
\begin{aligned}
    \int_{\Omega\setminus B(z,4d_0)}
    \Gamma_z^{2+\sigma}|D^3g_2|^2\,{\rm d}x
    &\lesssim
    \int_{4d_0}^{\operatorname{diam}(\Omega)}
    r^{\sigma-3} \,{\rm d}r
    \\
    &\quad+
    \sum_{|z-\omega|>2d_0}\int_{0}^{|z-\omega|}
    |z-\omega|^{-2-\sigma}
    \rho^{2\sigma-1}
    \,{\rm d}\rho\\
    &\lesssim
    h^{\sigma-2}.
\end{aligned}
\]
Together with the near-field estimate, this proves
\eqref{eq:weighted-D3-g2}.
\end{proof}

We now prove the two main theorems.

\medskip
\noindent{\bf Proof of Theorem~\ref{theorem conforming main results}.}
We first prove the \(W^{2,p}\)-stability estimate
\eqref{conforming W2,p}.  This part uses only the basic assumption \(\mu\ge2\). We still denote by $e:=g_2^z-g_{2,h}^z$.

For a fixed \(z\in\overline{\Omega}\), the definitions of \(g_2^z\) and \(g_{2,h}^z\), and Galerkin orthogonality give
\begin{equation}\label{eq:conforming-W2-representation}
    D^2u_h(z)
    =
    (D^2u_h,\delta_z)
    =
    (D^2u,\delta_z)
    -
    a(u,e).
\end{equation}
For \(p=\infty\), \eqref{eq:conforming-W2-representation} and
Theorem~\ref{theorem g-gh} directly imply
\[
    |D^2u_h(z)|
    \lesssim
    \seminorm{u}_{W^{2,\infty}(\Omega)}
    \left(
    \norm{\delta_z}_{L^1(T_z)}
    +
    \seminorm{e}_{W^{2,1}(\Omega)}
    \right)
    \lesssim
    \seminorm{u}_{W^{2,\infty}(\Omega)}.
\]
Taking the supremum over \(z\) proves \eqref{conforming W2,p} for
\(p=\infty\).

For \(2\le p<\infty\), we combine \eqref{eq:weighted-g2-error} with the standard weighted-norm argument of
Rannacher and Scott \cite{MR645661}. 
Applying the H{\"o}lder inequality to \eqref{eq:conforming-W2-representation}, we have
\begin{align} \label{weighted argument 1}
\begin{aligned}
\left| a(u,e) \right|
&\lesssim
\left( \int_{\Omega} \Gamma_z^{-2-\sigma} |D^2 u|^p \,\mathrm{d}x \right)^{\frac1p}
\left( \int_{\Omega} \Gamma_z^{-2-\sigma} \,\mathrm{d}x \right)^{\frac{p-2}{2p}}
\left( \int_{\Omega} \Gamma_z^{2+\sigma} |D^2 e|^2 \,\mathrm{d}x \right)^{\frac12}
\\
&\lesssim
\left(\sigma^{-1}h^{-\sigma}\right)^{\frac{p-2}{2p}} M_{\sigma,h}^z
\left( \int_{\Omega} \Gamma_z^{-2-\sigma} |D^2 u|^p \,\mathrm{d}x \right)^{\frac1p}.
\end{aligned}
\end{align}
Moreover, by \eqref{eq:delta-Lp-scaling}, 
\begin{equation}\label{weighted argument 2}
|(D^2u,\delta_z)|\lesssim h^{-\frac2p} \left(\int_{T_z}|D^2 u|^p \,\mathrm{d}x \right)^{\frac1p}.
\end{equation}
By combining \eqref{weighted argument 1} and \eqref{weighted argument 2}, integrating with respect to $z$, interchanging integration and using Corollary~\ref{corollary W2p}, we obtain
\begin{equation}
\begin{aligned}
\norm{D^2u_h}_{L^p(\Omega)}&\lesssim \left( h^{-2}\int_{\Omega}\int_{T_z} |D^2u|^p \,\mathrm{d}x\,\mathrm{d}z   \right)^{\frac1p}\\
&\quad +   \left(\sigma^{-1}h^{-\sigma}\right)^{\frac{p-2}{2p}} \sup_{z\in\overline{\Omega}}M_{\sigma,h}^z \left( \int_{\Omega}\int_{\Omega} \Gamma_z^{-2-\sigma} |D^2 u|^p \,\mathrm{d}x \,\mathrm{d}z \right)^{\frac1p}\\
&\lesssim \norm{D^2u}_{L^p(\Omega)}.
\end{aligned}
\end{equation}
This proves \eqref{conforming W2,p} for all \(2\le p\le\infty\).

We next prove the lower-order pointwise estimates.  Let \(i=0\) or \(i=1\),
and let \(g_i=g_i^z\) be the regularized Green's function from
\eqref{definition of regularized Green's functions}. Using the Galerkin orthogonality and the fact $\Pi_h g_i\in \mathring{V}_h$, we obtain for any $\chi_h\in  \mathring{V}_h$
\[
\begin{aligned}
    |D^i(u-u_h)(z)|
    &\le
    |D^i(u-\chi_h)(z)|
    +
    |D^i(\chi_h-u_h)(z)|
    \\
    &\lesssim
    \seminorm{u-\chi_h}_{W^{i,\infty}(\Omega)}
    +
    \seminorm{u-u_h}_{W^{2,\infty}(\Omega)}
    \seminorm{g_i-\Pi_hg_i}_{W^{2,1}(\Omega)} .
\end{aligned}
\]
Taking the infimum over \(\chi_h\in\mathring{V}_h\), together with the Galerkin orthogonality and the $W^{2,\infty}$-stability of $u_h$, gives
\begin{equation}\label{eq:conforming-lower-order-master}
\begin{aligned}
    \seminorm{u-u_h}_{W^{i,\infty}(\Omega)}
    \lesssim\;&
    \inf_{\chi_h\in\mathring{V}_h}
    \seminorm{u-\chi_h}_{W^{i,\infty}(\Omega)}
    \\
    &+
    \seminorm{g_i-\Pi_hg_i}_{W^{2,1}(\Omega)}
    \inf_{\chi_h\in\mathring{V}_h}
    \seminorm{u-\chi_h}_{W^{2,\infty}(\Omega)} .
\end{aligned}
\end{equation}

Assume now that \(\mu\ge3\).  By the interpolation estimate
\eqref{approximation of interpolation}, \eqref{g1 W41}, and \eqref{g0 W41},
\[
    \seminorm{g_1-\Pi_hg_1}_{W^{2,1}(\Omega)}
    \lesssim
    h^2\seminorm{g_1}_{W^{4,1}(\Omega)}
    \lesssim h,
\]
and
\[
    \seminorm{g_0-\Pi_hg_0}_{W^{2,1}(\Omega)}
    \lesssim
    h^2\seminorm{g_0}_{W^{4,1}(\Omega)}
    \lesssim h^2|\ln h|.
\]
Substituting these two estimates into
\eqref{eq:conforming-lower-order-master} proves
\eqref{conforming W1,infty} and \eqref{conforming Linfty-log}.

If \(\mu\ge4\), the logarithm in the \(L^\infty\)-estimate can be removed.
Indeed, using the near-field/far-field decomposition
\eqref{eq:near-singular-region} and the pointwise Green's function estimates
\eqref{pointwise1}--\eqref{pointwise3}, we have
\[
    |D^5g_0(x)|
    \lesssim
    |x-z|^{-3}
    +
    \sum_{\omega\in\mathcal{V}} |x-\omega|^{-3},
    \qquad x\in\Omega\setminus\Omega_0 .
\]
Therefore,
\[
\begin{aligned}
    \seminorm{g_0-\Pi_hg_0}_{W^{2,1}(\Omega)}
    &\lesssim
    h^2\seminorm{g_0}_{W^{4,1}(\Omega_0)}
    +
    h^3\seminorm{g_0}_{W^{5,1}(\Omega\setminus\Omega_0)}
    \\
    &\lesssim
    h^2
    +
    h^3
    \int_{\kappa h}^{\operatorname{diam}(\Omega)}
    r^{-3}r\,{\rm d}r
    \lesssim h^2 .
\end{aligned}
\]
Using this sharper bound in
\eqref{eq:conforming-lower-order-master} with \(i=0\) proves
\eqref{conforming Linfty}.\hfill\qedsymbol

\medskip
\noindent{\bf Proof of Theorem~\ref{theorem nonconforming main results}.}
We first prove the \(W^{2,p}\)-estimate. Let \(\chi_h\in\mathring{V}_h\).
In view of the definitions of $g_2$, $g_{2,h}$, $u$ and $u_h$, straightforward algebraic manipulations lead to
\begin{equation}\label{eq:nonconforming-D2-representation}
\begin{aligned}
    D^2(\chi_h-u_h)(z)
    &=
    (\chi_h-u_h,D^2\delta_z)
    =
    a_h(\chi_h-u_h,g_{2,h})
    \\
    &=
    a_h(u-\chi_h,g_2-g_{2,h})
    +
    \left[
    (\Delta^2g_2,u-\chi_h)
    -
    a_h(g_2,u-\chi_h)
    \right]
    \\
    &\quad
    -
    (\Delta^2g_2,u-\chi_h)
    +
    \left[
    (\Delta^2u,g_2-g_{2,h})
    -
    a_h(u,g_2-g_{2,h})
    \right].
\end{aligned}
\end{equation}
We choose \(\chi_h=\Pi_hu\).  The first term on the right-hand side of
\eqref{eq:nonconforming-D2-representation} is estimated by
Corollary~\ref{corollary W2p}; the two bracketed terms are estimated by combining Morley consistency estimates \eqref{split-version Morley consistency estimate} and the weighted-norm argument used in the proof of Theorem~\ref{theorem conforming main results}. This gives, for \(2\le p<\infty\),
\begin{equation}\label{eq:nonconforming-W2-intermediate}
\begin{aligned}
    \seminorm{\Pi_hu-u_h}_{W^{2,p}(\Omega)}
    \lesssim
    \seminorm{u-\Pi_hu}_{W^{2,p}(\Omega)}+h \seminorm{u}_{W^{3,p}(\Omega)}+h^2
    \norm{f}_{L^p(\Omega)} .
\end{aligned}
\end{equation}
For clarity of exposition, a detailed proof of \eqref{eq:nonconforming-W2-intermediate} is deferred to Appendix~\ref{supp:Morley}. 
Using the Morley interpolation estimate $\seminorm{u-\Pi_hu}_{W^{2,p}(\Omega)}
    \lesssim
    h\seminorm{u}_{W^{3,p}(\Omega)}$,
we obtain
\[
    \seminorm{\Pi_hu-u_h}_{W^{2,p}(\Omega)}
    \lesssim
    h\seminorm{u}_{W^{3,p}(\Omega)}
    +
    h^2\norm{f}_{L^p(\Omega)} 
\]
for $2\leq p<\infty$. For $p=\infty$, by \eqref{eq:nonconforming-D2-representation},  \eqref{split-version Morley consistency estimate}, \eqref{eq:delta-Lp-scaling} and \eqref{g2-g2h}, we have
\begin{equation*}
 \seminorm{\Pi_hu-u_h}_{W^{2,\infty}(\Omega)}
    \lesssim
    h\seminorm{u}_{W^{3,\infty}(\Omega)}
    +
    h^2\norm{f}_{L^{\infty}(\Omega)}.    
\end{equation*}
Thus, using the interpolation estimate for \(u-\Pi_hu\), we arrive at
\eqref{nonconforming W2,p}.

We next prove the \(W^{1,\infty}\)-estimate.  Let \(g_1=g_1^z\) be the
regularized first-derivative Green's function. Similarly as \eqref{eq:nonconforming-D2-representation},  straightforward algebraic manipulations yield 
\begin{equation}\label{prove non 1infty}
\begin{aligned}
    |D(u-u_h)(z)|
    \lesssim\;&
   \underbrace{ |D(u-\Pi_hu)(z)|}_{:=J_1}
    +
    \underbrace{ |(\Pi_hu-u,D\delta_z)|}_{:=J_2}+ \underbrace{ |a_h(u-u_h,g_1-\Pi_hg_1)|}_{:=J_3}
    \\
    &+
    \underbrace{ \left|
    (u-u_h,\Delta^2g_1)-a_h(u-u_h,g_1)
    \right|}_{:=J_4}
    \\
    &+
   \underbrace{ \left|
    a_h(u,g_1-\Pi_hg_1)
    -
    (f,g_1-\Pi_hg_1)
    \right|}_{:=J_5} .
\end{aligned}
\end{equation}
The first three terms are estimated using interpolation error bounds, while the last two are controlled by consistency error estimates. More precisely, by \eqref{nonconforming W2,p}, \eqref{g1 W31}, the approximation property of $\Pi_h$ and the scaling properties
\eqref{eq:regularized-delta-properties}, we have
\begin{align*}
J_i &\lesssim h^2|\ln h|(\seminorm{u}_{W^{3,\infty}(\Omega)}+h\norm{f}_{L^{\infty}(\Omega)})+h^3 \seminorm{u}_{W^{3,\infty}(\Omega)}\norm{D\delta_z}_{L^1(T_z)}\\
&\lesssim   h^2|\ln h|(\seminorm{u}_{W^{3,\infty}(\Omega)}+h\norm{f}_{L^{\infty}(\Omega)})
\end{align*}
for $i=1,2,3$. For $J_4$, we use \eqref{split-version Morley consistency estimate} with $p=1$ and $D=\Omega$, to get
\begin{equation*}
J_4\lesssim \left(h\seminorm{g_1}_{W^{3,1}(\Omega)}+h^2\norm{D\delta_z}_{L^1(T_z)} \right)\seminorm{u-u_h}_{W_h^{2,\infty}(\Omega)}, 
\end{equation*}
which, together with \eqref{nonconforming W2,p} and \eqref{g1 W31}, yields
\begin{equation*}
 J_4\lesssim  h^2|\ln h| \left(\seminorm{u}_{W^{3,\infty}(\Omega)}+h\norm{f}_{L^{\infty}(\Omega)}\right). 
\end{equation*}
The estimate for $J_5$ follows from the similar argument for estimating $J_4$,
\begin{align*}
 J_5 &\lesssim  \left(h\seminorm{u}_{W^{3,\infty}(\Omega)}+h^2\norm{f}_{L^{\infty}(\Omega)} \right)\seminorm{g_1-\Pi_h g_1}_{W^{2,1}(\Omega)}\\
 &\lesssim h^2|\ln h|\left(\seminorm{u}_{W^{3,\infty}(\Omega)}+h\norm{f}_{L^{\infty}(\Omega)} \right).
\end{align*}
Combining above estimates for $J_i$ ($1\leq i\leq 5$) and taking the supremum over \(z\in\overline{\Omega}\) lead to \eqref{nonconforming W1,infty}.

Finally, we prove the \(L^\infty\)-estimate.  By the same
representation and consistency argument as above, for  any fixed $2< p<\infty$ and
\(p'\) satisfying \(1/p+1/p'=1\),
\[
\begin{aligned}
    |(u-u_h)(z)|
    \lesssim\;&
    \norm{u-\Pi_hu}_{L^{\infty}(T_z)}+
    \left(
    h\seminorm{g_0}_{W^{3,p'}(\Omega)}
    +
    h^2\norm{\delta_z}_{L^{p'}(T_z)}\right)\seminorm{u-u_h}_{W^{2,p}(\Omega)}   
    \\
    &+
    h\seminorm{u}_{W^{3,p}(\Omega)}\seminorm{g_0-\Pi_h g_0}_{W_h^{2,p'}(\Omega)}
    +
    h^2\norm{f}_{L^{p}(\Omega)}\seminorm{g_0-\Pi_h g_0}_{W_h^{2,p'}(\Omega)}.
\end{aligned}
\]
The regularity estimate for the
regularized Green's function gives (see \cite{MR775683})
\[
    \seminorm{g_0}_{W^{3,q}(\Omega)}
    \lesssim
    \norm{\delta_z}_{W^{-1,q}(\Omega)}
    \lesssim 1  \quad  \forall 1<q<2.
\]
Hence, by the Sobolev embedding $W^{3,p}(\Omega)\hookrightarrow W^{2,\infty}(\Omega)$, 
\[
    |(u-u_h)(z)|
    \lesssim
    h^2
    \left(
    \seminorm{u}_{W^{3,p}(\Omega)}
    +
    h\norm{f}_{L^p(\Omega)}
    \right)
\]
for $2<p<\infty$. The validity for $p=\infty$ follows from the Sobolev embedding. Taking the supremum over \(z\) proves \eqref{nonconforming Linfty}.\hfill\qedsymbol


\section{Conclusion}\label{conclusion}
Sharp pointwise error estimates for finite element approximations of the clamped biharmonic equation are obtained on convex polygonal domains. More precisely, for conforming elements, we establish $W^{2,p}$-stability for all $2 \le p \le \infty $, together with best-approximation
type estimates in $W^{1,\infty}$- and $L^{\infty}$-norms. For the nonconforming Morley elements, we give optimal
estimates in $W^{2,p}$-, $W^{1,\infty}$-, and $L^{\infty}$-norms. Such sharp pointwise error estimates remain open for polyhedra and nonconvex polygons. An extension of the present results to three-dimensional polyhedral
domains is not straightforward. The analysis relies essentially on the
H{\"o}lder estimates for mixed derivatives of the clamped biharmonic
Green's function, together with the a priori estimates for regularized Green's functions.
In three dimensions, the edge and
vertex singularities of polyhedral domains make the corresponding
Green's function and global regularity estimates substantially more
delicate. On the discrete side, standard three-dimensional
low-order nonconforming elements do not automatically satisfy the
vertex-continuity and conforming-companion hypotheses in
Assumption~\ref{assumption}. Consequently, the present argument does not extend directly to three dimensions.

The sharpness of the remaining logarithmic factors in \eqref{conforming Linfty-log} and \eqref{nonconforming W1,infty} is also unresolved. For second-order elliptic
problems discretized by linear finite elements, the lower bounds show that the logarithmic factor in maximum-norm error estimates cannot in general be removed; see \cite{MR847015} for two-dimensional domains and \cite{MR3925475} for three-dimensional domains. Whether the logarithmic factors in
\eqref{conforming Linfty-log} for $\mu=3$ and \eqref{nonconforming W1,infty} are likewise unavoidable
remains open.

\appendix

\section{Verification of Assumption~\ref{assumption}}
\label{supp:assumption}
In this section, we provide further clarification of Assumption~\ref{assumption}.

\subsection{Choice of the interpolation operator}

For some conforming finite elements, including the Argyris, Bell, and Bogner--\allowbreak Fox--\allowbreak Schmit elements, the degrees of freedom involve second-order derivatives. Consequently, the associated canonical interpolation operators do not, in general, satisfy Assumption~(A1). We therefore employ averaging-based quasi-interpolation operators instead. More precisely, for the Argyris, Bell, and Bogner--Fox--Schmit elements, one may
use the boundary-preserving Hermite-type quasi-interpolation
operator constructed by Girault and Scott~\cite{MR1898745}.
This operator is a projection onto the corresponding globally
$C^1$ finite element space, preserves the polynomial traces required
by the homogeneous clamped boundary conditions, and possesses
optimal local approximation properties. Thus, Assumption~(A1) is satisfied for these elements. 
Earlier boundary-preserving
quasi-interpolation constructions specifically for the Bell and
bicubic Bogner--Fox--Schmit elements can also be found
in~\cite{MR1388488}.

For elements involving derivatives of order at most one in their degrees of freedom, the canonical finite element interpolation operator can be used directly, owing to the two-dimensional embedding $W^{3,1}(T) \hookrightarrow C^1(\overline{T}).$
One may verify that these canonical finite element interpolation operators satisfy (A1).

\subsection{Verification of (A2)}
In this subsection, we  verify the superapproximation property (A2) for some types of finite elements. For conforming elements, it follows from the argument in \cite[Lemma 3.1]{MR4166452}. However, for certain nonconforming elements, such as the Morley element, $\Pi_h(\eta^4v_h)$ is not well defined. Consequently, the proof of this lemma cannot rely solely on the approximation property (A1). For simplicity, we verify Assumption~(A2) below only for the Morley element; the other elements can be treated analogously.

\medskip
\noindent{\bf Verification of (A2) for the triangular Morley element.}
Let \(V_h\) be the triangular Morley element space and let \(\Pi_h\) be the canonical finite element interpolation operator for the Morley element. Moreover, for each $T\in\mathcal{T}_h$, let $\Pi_T$ denote the elementwise
counterpart of $\Pi_h$. For every $v_h\in V_h$, we first prove
\begin{equation}\label{1 superapproximation}
      \seminorm{\eta^4v_h-\Pi_T(\eta^4v_h)}_{H^2(T)}
   \lesssim
   h d^{-1}
   \left(   \seminorm{\eta^2v_h}_{H^2(T)}
   +d^{-1}\seminorm{v_h}_{H^1(T)}
   +d^{-2}\norm{v_h}_{L^2(T)}
   \right).
\end{equation}
Using the interpolation estimate  and
the bounds on \(\eta\), we obtain
\begin{align}\label{super Morley interpolation}
\begin{aligned}
 \seminorm{\eta^4v_h-\Pi_T(\eta^4v_h)}_{H^2(T)}
 &=
 \seminorm{(\eta^4-P_T^0(\eta^4))v_h
 -\Pi_T((\eta^4-P_T^0(\eta^4))v_h)}_{H^2(T)}
 \\
 &\lesssim
 h
 \seminorm{(\eta^4-P_T^0(\eta^4))v_h}_{H^3(T)}
 \\
 &\lesssim
 h d^{-1}
 \left(
 d^{-2}\norm{v_h}_{L^2(T)}
 +d^{-1}\seminorm{v_h}_{H^1(T)}
 \right)
 \\
 &\quad
 +h d^{-1}\norm{\eta^2}_{L^\infty(T)}
\seminorm{v_h}_{H^2(T)}.
\end{aligned}
\end{align}
The last term is localized by inserting and subtracting
\(\norm{\eta^2}_{L^\infty(T)}v_h\).  This gives
\[
 \norm{\eta^2}_{L^\infty(T)}\seminorm{v_h}_{H^2(T)}
 \lesssim
 \seminorm{\eta^2v_h}_{H^2(T)}
 +d^{-1}\seminorm{v_h}_{H^1(T)}
 +d^{-2}\norm{v_h}_{L^2(T)},
\]
which leads to \eqref{1 superapproximation}, together with \eqref{super Morley interpolation}. 
Note that the derivation of estimate~\eqref{1 superapproximation} can be simplified by using the fact that $|v_h|_{H^3(T)}=0.$
However, for finite elements of piecewise polynomial degree greater than $\mu$,
the estimate should  be derived as in \eqref{super Morley interpolation}.

\par We now prove \eqref{assumption pi}. Let $v_h\in \mathring{V}_h$. Define $\tilde{v}_h\in P_2(\Omega,\mathcal{T}_h)$ by $\tilde{v}_h|_T:=\Pi_T(\eta^4v_h)$. It suffices to show that there exists $w_h\in \mathring{V}_h(U_2)$ such that
\begin{equation}\label{2 superapproximation}
    \seminorm{\tilde{v}_h-w_h}_{H_h^2(U_2)}
    \lesssim
    h\Bigl(
    d^{-1}\seminorm{\eta^2 v_h}_{H_h^2(U_2)}
    +d^{-2}\seminorm{v_h}_{H_h^1(U_2)}
    +d^{-3}\norm{v_h}_{L^2(U_2)}
    \Bigr).
\end{equation}
Since $\tilde{v}_h$ may fail to belong
to \(V_h\), we introduce a modified interpolation \(\widetilde{\Pi}_h\), in the
spirit of \cite{MR1229633}.  On each triangle \(T\), the function
\(\widetilde{\Pi}_h(\eta^4v_h)\) is defined by the following Morley degrees of
freedom: $\widetilde{\Pi}_h(\eta^4v_h)(A_i)
    =
    \eta^4(A_i)v_h(A_i),
    \ i=1,2,3,$
at the vertices \(A_i\) of \(T\), and
\[
    \frac{\partial \widetilde{\Pi}_h(\eta^4v_h)}{\partial n}(m_i)
    =
    \frac{\partial(\eta^4)}{\partial n}(m_i)\,I_h^1v_h(m_i)
    +
    \eta^4(m_i)\frac{\partial v_h}{\partial n}(m_i),
    \qquad i=1,2,3,
\]
at the midpoints \(m_i\) of the edges.  By construction,
\(\widetilde{\Pi}_h(\eta^4v_h)\in \mathring{V}_h(U_2)\). Take $w_h=\widetilde{\Pi}_h(\eta^4v_h)$. For any \(p_T\in P_2(T)\),
the norm equivalence for Morley degrees of freedom gives $\norm{p_T}_{L^\infty(T)}
    \lesssim
    \sum_{i=1}^3 |p_T(A_i)|
    +
    h_T\sum_{i=1}^3
    \left|\frac{\partial p_T}{\partial n}(m_i)\right|$.
Thus, by the definitions of \(\Pi_h\) and
\(\widetilde{\Pi}_h\), and standard inverse estimates, we obtain
\begin{align*}
    \norm{\widetilde{v}_h-w_h}_{L^\infty(T)}
    &
    \lesssim   h\,\norm{D\eta}_{L^\infty(T)}
    \norm{\eta^3(v_h-I_h^1v_h)}_{L^\infty(T)}
    \\
    &
    \lesssim
    h d^{-1}
    \Bigl(
    \norm{\eta^2v_h-I_h^1(\eta^2v_h)}_{L^\infty(T)}
    +
    \norm{I_h^1(\eta^2v_h)-\eta^2I_h^1v_h}_{L^\infty(T)}
    \Bigr)
    \\
    & \lesssim
    h^2d^{-1}\seminorm{\eta^2v_h}_{H^2(T)}
    +  h^2d^{-2}\seminorm{v_h}_{H^1(T)}
    +   h^2d^{-3}\norm{v_h}_{L^2(T)}.
\end{align*}
After scaling, applying inverse estimates and summing over \(T\subset U_2\), this proves (A2). Notice that if  \(U_2\cap\partial\Omega=\varnothing\), then $\widetilde{\Pi}_h(\eta^4v_h)\in \mathring{V}_h(U_2)$ for $v_h\in V_h$. Thus, \eqref{assumption pi} still holds for all $v_h\in V_h$.

\subsection{Verification of (A3)}
Assumption~(A3) is automatically valid by taking $v=v_h$ for conforming elements. Accordingly, this section is mainly devoted to verifying condition~(A3) for nonconforming finite elements.
For simplicity, we verify only for the triangular Morley element.

\medskip
\noindent{\bf Verification of (A3) for the triangular Morley element.} To verify (A3), let \(E_h\) be the averaging, or companion, operator from the
Morley space into the Hsieh--Clough--Tocher \(C^1\) macroelement space, as in
\cite{MR2670114}. For \(w:=v_h-E_hv_h\), the Morley degrees of freedom imply that,
on each \(T\),
\[
\norm{w}_{L^\infty(T)}
\lesssim
\sum_{i=1}^3 |w(A_i)|
+h_T\sum_{i=1}^3|\nabla w(A_i)|
+h_T\sum_{i=1}^3
\left|\frac{\partial w}{\partial n}(m_i)\right|
=
h_T\sum_{i=1}^3|\nabla w(A_i)|.
\]
Moreover, the normal and tangential derivatives of \(w\) have zeros on each
edge in the sense of the Morley degrees of freedom.  By inverse estimates and
standard scaling on element patches \(\omega_T\), one obtains
\[
    \seminorm{v_h-E_hv_h}_{H^j(T)}
    \lesssim
    h_T^{2-j}\seminorm{v_h}_{H_h^2(\omega_T)},
    \qquad j=0,1,2.
\]
Summing over all elements in \(U\) gives \eqref{assumption 4} with
\(v=E_hv_h\).

\section{\texorpdfstring{$W^{4,p}$}{W4,p}-regularity estimate}\label{supp:regularity}
In this section, we derive the $W^{4,p}$- regularity estimate \eqref{W4p regulatity}, based on \cite[Section 7.3]{MR775683}.

\noindent{\bf Proof of \eqref{W4p regulatity}.} According to \cite[Theorem~7.3.2.1, Remark 7.3.2.3]{MR775683}, it suffices to show that there exists a $p_0>1$ such that, for every $1<p<p_0$,
the following characteristic equation 
\begin{equation}\label{eq:characteristic equation}
\sinh^2{(\lambda \omega)}=\lambda^2 \sin^2{\omega} 
\end{equation}
has no root in the strip $-(1+2/p')\leq \operatorname{Im}\lambda<0$ (other than $-\mathrm{i}$), where $0<\omega<\pi$ is an arbitrary interior angle of $\Omega$ and $1/p+1/p'=1$. Once this regularity is established, the estimate follows from
\cite[Corollary 7.3.1.9]{MR775683}.

Define
\[
F_\omega(\lambda)
:=
\sinh^2(\omega\lambda)-\lambda^2\sin^2\omega .
\]
Since
\[
F_\omega(\lambda)
=
\bigl(\sinh(\omega\lambda)-\lambda\sin\omega\bigr)
\bigl(\sinh(\omega\lambda)+\lambda\sin\omega\bigr),
\]
every root satisfies
\[
\sinh(\omega\lambda)
=
\hat{\sigma}\lambda\sin\omega,
\qquad
\hat{\sigma}\in\{1,-1\}.
\]
Write
\[
\lambda=\xi-\mathrm{i}\hat{\xi},
\qquad \hat{\xi}>0.
\]
Choose $\varepsilon_0>0$ sufficiently small so that
\[
(1+\varepsilon_0)\omega<\pi.
\]
For $0<\hat{\xi}<1+\varepsilon_0$, comparison of the imaginary parts gives
\[
\cosh(\omega\xi)\sin(\hat{\xi}\omega)
=
\hat\sigma\hat{\xi}\sin\omega.
\]
Since $\sin(\hat{\xi}\omega)>0$ and $\sin\omega>0$, hence $\hat{\sigma}=1$ and 
\begin{equation}
\cosh(\omega\xi)\sin(\hat{\xi}\omega)
=
\hat{\xi}\sin\omega.
\label{eq:corner-eigenvalue-imaginary}
\end{equation}

If $0<\hat{\xi}<1$, the strict concavity of $\sin t$ on $(0,\pi)$ yields
\[
\sin(\hat{\xi}\omega)>\hat{\xi}\sin\omega.
\]
Together with $\cosh(\omega\xi)\ge1$, this contradicts
\eqref{eq:corner-eigenvalue-imaginary}. Thus no root exists for
$0<\hat{\xi}<1$.

If $\hat{\xi}=1$, \eqref{eq:corner-eigenvalue-imaginary} gives
\[
\cosh(\omega\xi)=1,
\]
and consequently $\xi=0$. Hence $\lambda=-\mathrm{i}$ is the only
root with $\hat{\xi}=1$.

It remains to exclude roots with $\hat{\xi}>1$ sufficiently close to $1$.
Suppose, to the contrary, that there exists a sequence of roots
\[
\lambda_n=\xi_n-\mathrm{i}\hat{\xi}_n,
\qquad
1<\hat{\xi}_n<1+\varepsilon_0,\qquad
\hat{\xi}_n\to1.
\]
By \eqref{eq:corner-eigenvalue-imaginary},
\[
\cosh(\omega\xi_n)
=
\frac{\hat{\xi}_n\sin\omega}{\sin(\hat{\xi}_n\omega)}
\longrightarrow1,
\]
and therefore $\xi_n\to0$. Hence
\[
\lambda_n\to-\mathrm{i}.
\]

Set
\[
H_\omega(\lambda)
=
\sinh(\omega\lambda)-\lambda\sin\omega.
\]
We have $H_\omega(-\mathrm{i})=0$, while
\[
H_\omega'(-\mathrm{i})
=
\omega\cos\omega-\sin\omega\neq0,
\]
because
\[
\sin\omega-\omega\cos\omega>0,
\qquad 0<\omega<\pi.
\]
Thus $-\mathrm{i}$ is a simple, and hence isolated, zero of
$H_\omega$, contradicting $\lambda_n\to-\mathrm{i}$ with
$\lambda_n\neq-\mathrm{i}$.

Therefore there exists $0<\delta_\omega<\varepsilon_0$ such that the stated strip $-(1+\delta_{\omega})<\operatorname{Im}\lambda<0$
contains no root other than $-\mathrm{i}$. Choosing $p_0$ satisfying $2/p'_0<\delta_{\omega}$ for every $\omega$ gives \eqref{W4p regulatity}. \hfill\qedsymbol

\section{Consistency error estimate for the Morley element}\label{supp:Morley}
The consistency error estimate for the Morley element can be found in \cite{Z.Shi1}. For completeness, and in view of the different form required in this paper, we provide the details of the proofs of \eqref{split-version Morley consistency estimate} and \eqref{eq:nonconforming-W2-intermediate} below.

\medskip
\noindent{\bf Proof of \eqref{split-version Morley consistency estimate}.}
We use the notations in Lemma~\ref{lemma Morley consistency estimate}.
Since $w$ is continuous at each vertex, the following identity holds (see (17) in \cite{Z.Shi1})
\begin{align}\label{consistency expression}
\begin{aligned}
a_h(u,w)-(f,w)&= (f,I_h^1w-w)+ \underbrace{ \sum_{T\in\mathcal{T}_h} \int_T \nabla\Delta u\cdot\nabla (I_h^1w-w)\,\mathrm{d}x }_{=:Q_2} \\
&\quad +\underbrace{ \sum_{T\in\mathcal{T}_h} \int_{\partial T} \left( \Delta u-\frac{\partial^2u}{\partial \tau^2} \right) \frac{\partial w}{\partial n}\,\mathrm{d}\ell }_{=:Q_3}
    +
\underbrace{ \sum_{T\in\mathcal{T}_h} \int_{\partial T} \frac{\partial^2u}{\partial \tau\,\partial n} \frac{\partial w}{\partial \tau}\,\mathrm{d }\ell }_{=:Q_4} ,     
\end{aligned}
\end{align}
where $\tau$ denotes the unit tangent vector.
We now estimate \(Q_i\), (\(i=2,3,4\)), in turn. 

\par For $Q_2$, since $\nabla \Delta I_h^2u$ vanishes piecewise, by using  trace inequalities and  integrating by parts, one may deduce
\begin{align*}
|Q_2|&=|\sum_{T\in\mathcal{T}_h}\int_T \nabla\Delta (u-I_h^2u)\cdot\nabla (I_h^1w-w) \,{\rm d}x|\\  
&\lesssim h\seminorm{u}_{W^{3,p}(D\cap \Omega_w)}\seminorm{w}_{W_h^{2,p'}(D\cap \Omega_w)}+h^{\sigma}\seminorm{u}_{C^{2,\sigma}(\Omega_w\backslash D)}\seminorm{w}_{W^{2,1}_h(\Omega_w\backslash D)}.
\end{align*}
Estimates for $Q_3$ and $Q_4$ follow from the standard argument for the nonconforming element. More precisely, for $i=3,4$, we have
\begin{equation*}
    |Q_i|\lesssim h\seminorm{u}_{W^{3,p}(D\cap \Omega_w)}\seminorm{w}_{W_h^{2,p'}(D\cap \Omega_w)}+h^{\sigma}\seminorm{u}_{C^{2,\sigma}(\Omega_w\backslash D)}\seminorm{w}_{W^{2,1}_h(\Omega_w\backslash D)}.
\end{equation*}
Combining the above estimates for \(Q_i\), ($2\leq i\leq 4$), yields the desired result \eqref{split-version Morley consistency estimate}. \hfill\qedsymbol

\medskip
\noindent{\bf Proof of \eqref{eq:nonconforming-W2-intermediate}.} 
In view of \eqref{eq:nonconforming-D2-representation}, we have
\begin{align*}
\begin{aligned}
    |D^2(\Pi_hu-u_h)(z)|
    &\leq 
    \underbrace{\left|a_h(u-\Pi_hu,g_2-g_{2,h})\right|}_{:=K_1}
    +
    \underbrace{\left|
    (\Delta^2g_2,u-\Pi_hu)
    -
    a_h(g_2,u-\Pi_hu)
    \right|}_{:=K_2}
    \\
    &\quad    
    +\underbrace{|(\Delta^2g_2,u-\Pi_hu)|}_{K_3}
    +
    \underbrace{\left|
    (\Delta^2u,g_2-g_{2,h})
    -
    a_h(u,g_2-g_{2,h})
    \right|}_{:=K_4}.
\end{aligned}   
\end{align*}
We estimate $K_i$ ($1\le i\le 4$) by combining the consistency error estimate and the weighted-norm argument used in the proof of Theorem~\ref{theorem conforming main results}. We still use $e$ to denote $g_2-g_{2,h}$.

We first estimate $K_1$. By the H{\"o}lder inequality, we obtain
\begin{align}\label{estimate for K1}
\begin{aligned}
K_1
&\lesssim
\left( \int_{\Omega} \Gamma_z^{-2-\sigma} |D^2 (u-\Pi_hu)|^p
\,\mathrm{d}x \right)^{\frac1p}
\left( \int_{\Omega} \Gamma_z^{-2-\sigma} \,\mathrm{d}x
\right)^{\frac{p-2}{2p}}
\left( \int_{\Omega} \Gamma_z^{2+\sigma} |D^2 e|^2
\,\mathrm{d}x \right)^{\frac12}
\\
&\lesssim
\left(\sigma^{-1}h^{-\sigma}\right)^{(p-2)/(2p)} M_{\sigma,h}^z
\left( \int_{\Omega} \Gamma_z^{-2-\sigma} |D^2 (u-\Pi_hu)|^p
\,\mathrm{d}x \right)^{1/p}.
\end{aligned}
\end{align}
Similarly,
\begin{equation}\label{estimate for K2}
\begin{aligned}
K_3
\lesssim\;&
\left(\sigma^{-1}h^{-\sigma}\right)^{\frac{p-2}{2p}}
\left( \int_{T_z} \Gamma_z^{2+\sigma} \delta_z^2
\,\mathrm{d}x \right)^{\frac12}
\left( \int_{\Omega} \Gamma_z^{-2-\sigma} |D^2 (u-\Pi_hu)|^p
\,\mathrm{d}x \right)^{\frac1p},
\end{aligned}
\end{equation}
where we use the fact $(\Delta^2 g_2,u-\Pi_hu)=(\delta, D^2(u-\Pi_hu))$. 

For $K_2$, use the split as in  \eqref{consistency expression}, apply weighted-norm argument to the four terms as in \eqref{estimate for K1} and combine with the trace inequality and the interpolation approximation,  which yields
\begin{align}\label{estimate for K3}
\begin{aligned}
K_2&\lesssim  \left(\sigma^{-1}h^{-\sigma}\right)^{\frac{p-2}{2p}}  \left( \int_{\Omega} \Gamma_z^{-2-\sigma} |D^2 (u-\Pi_hu)|^p \,\mathrm{d}x \right)^{\frac1p} \\
&\quad \times \left[\left( \int_{T_z} \Gamma_z^{2+\sigma} \delta_z^2 \,\mathrm{d}x \right)^{\frac12}+h\left( \int_{\Omega} \Gamma_z^{2+\sigma} |D^3g_2|^2 \,\mathrm{d}x \right)^{\frac12} \right]. 
\end{aligned}
\end{align}
Similarly, for $K_4$, we have
\begin{align}\label{estimate for K4}
\begin{aligned}
K_4 &\lesssim \left(\sigma^{-1}h^{-\sigma}\right)^{\frac{p-2}{2p}} (h^2M_{\sigma,h}^z) \left( \int_{\Omega} \Gamma_z^{-2-\sigma} |f|^p \,\mathrm{d}x \right)^{\frac1p}\\
&\quad + \left(\sigma^{-1}h^{-\sigma}\right)^{\frac{p-2}{2p}} hM_{\sigma,h}^z \left( \int_{\Omega} \Gamma_z^{-2-\sigma} |D^3 u|^p \,\mathrm{d}x \right)^{\frac1p},
\end{aligned}    
\end{align}
where the trace inequality and the interpolation approximation are used.

Finally, collecting \eqref{estimate for K1}---\eqref{estimate for K4}, integrating with respect to $z$, interchanging the order of integration, and using Corollary~\ref{corollary W2p} and Lemma~\ref{lemma weighted D3}, gives \eqref{eq:nonconforming-W2-intermediate}. \hfill\qedsymbol

\bibliographystyle{plain}
\bibliography{reference}

@article{NitscheSchatz1974,
  author  = {Nitsche, J. A. and Schatz, A. H.},
  title   = {Interior estimates for {R}itz--{G}alerkin methods},
  journal = {Math. Comp.},
  volume  = {28},
  number  = {128},
  pages   = {937--958},
  year    = {1974},
  issn    = {0025-5718}
}

@article{BrambleNitscheSchatz1975,
  author  = {Bramble, J. H. and Nitsche, J. A. and Schatz, A. H.},
  title   = {Maximum-norm interior estimates for {R}itz--{G}alerkin methods},
  journal = {Math. Comp.},
  volume  = {29},
  number  = {131},
  pages   = {677--688},
  year    = {1975},
  issn    = {0025-5718}
}

@article{SchatzWahlbin1977,
  author   = {Schatz, A. H. and Wahlbin, L. B.},
  title    = {Interior maximum norm estimates for finite element methods},
  journal  = {Math. Comp.},
  volume   = {31},
  number   = {138},
  pages    = {414--442},
  year     = {1977},
  issn     = {0025-5718},
  mrnumber = {0431753},
  doi      = {10.1090/S0025-5718-1977-0431753-X}
}

@article{SchatzWahlbin1978,
  author   = {Schatz, A. H. and Wahlbin, L. B.},
  title    = {Maximum norm estimates in the finite element method on plane polygonal domains. {I}},
  journal  = {Math. Comp.},
  volume   = {32},
  number   = {141},
  pages    = {73--109},
  year     = {1978},
  issn     = {0025-5718},
  mrnumber = {0502065},
  doi      = {10.1090/S0025-5718-1978-0502065-1}
}

@article{Scott,
  author   = {Scott, L. Ridgway},
  title    = {Optimal {$L^{\infty}$} estimates for the finite element method on irregular meshes},
  journal  = {Math. Comp.},
  volume   = {30},
  number   = {136},
  pages    = {681--697},
  year     = {1976},
  issn     = {0025-5718},
  mrnumber = {0436617},
  doi      = {10.1090/S0025-5718-1976-0436617-2}
}

@article {MR2728969,
    AUTHOR = {Demlow, A. and Guzm\'{a}n, J. and Schatz, A. H.},
     TITLE = {Local energy estimates for the finite element method on
              sharply varying grids},
   JOURNAL = {Math. Comp.},
    VOLUME = {80},
      YEAR = {2011},
    NUMBER = {273},
     PAGES = {1--9},
      ISSN = {0025-5718},
   MRCLASS = {65N30 (65N15)},
  MRNUMBER = {2728969},
MRREVIEWER = {Christian Wieners},
       DOI = {10.1090/S0025-5718-2010-02353-1},
}

@article {MR1229633,
    AUTHOR = {Wang, M.},
     TITLE = {{$L^\infty$} error estimates of nonconforming finite elements
              for the biharmonic equation},
   JOURNAL = {J. Comput. Math.},
    VOLUME = {11},
      YEAR = {1993},
    NUMBER = {3},
     PAGES = {276--288},
      ISSN = {0254-9409},
   MRCLASS = {65N30},
  MRNUMBER = {1229633},
}

@article{Z.Shi1,
author = {Shi, Z.},
title = {ON THE ERROR ESTIMATES OF {M}ORLEY ELEMENT},
publisher = {Mathematica Numerica Sinica},
year = {1990},
journal = {Math. Numer. Sin.},
volume = {12},
number = {2},
pages = {113-118},
doi = {10.12286/jssx.1990.2.113},
}

@article {MR4166452,
    AUTHOR = {Leykekhman, D.},
     TITLE = {Pointwise error estimates for {$C^0$} interior penalty
              approximation of biharmonic problems},
   JOURNAL = {Math. Comp.},
    VOLUME = {90},
      YEAR = {2021},
    NUMBER = {327},
     PAGES = {41--63},
      ISSN = {0025-5718},
   MRCLASS = {65N30 (65N15)},
  MRNUMBER = {4166452},
       DOI = {10.1090/mcom/3596},
}

@article {MR2670114,
    AUTHOR = {Brenner, S. C. and Gudi, T. and Sung, L.-Y.},
     TITLE = {An a posteriori error estimator for a quadratic
              {$C^0$}-interior penalty method for the biharmonic problem},
   JOURNAL = {IMA J. Numer. Anal.},
    VOLUME = {30},
      YEAR = {2010},
    NUMBER = {3},
     PAGES = {777--798},
      ISSN = {0272-4979},
   MRCLASS = {65N30 (35J40 65N15)},
  MRNUMBER = {2670114},
MRREVIEWER = {Jeffrey S. Ovall},
       DOI = {10.1093/imanum/drn057},
}

@article {MR1143406,
    AUTHOR = {Maz'ya, V. G. and Rossmann,
              J.},
     TITLE = {On the {A}gmon-{M}iranda maximum principle for solutions of
              elliptic equations in polyhedral and polygonal domains},
   JOURNAL = {Ann. Global Anal. Geom.},
    VOLUME = {9},
      YEAR = {1991},
    NUMBER = {3},
     PAGES = {253--303},
      ISSN = {0232-704X},
   MRCLASS = {35B50 (35J40)},
  MRNUMBER = {1143406},
MRREVIEWER = {Cornelius\ O.\ Horgan},
       DOI = {10.1007/BF00136815},
}

@article {MR2988473,
    AUTHOR = {Hu, J. and Shi, Z.},
     TITLE = {The best {$L^2$} norm error estimate of lower order finite
              element methods for the fourth order problem},
   JOURNAL = {J. Comput. Math.},
    VOLUME = {30},
      YEAR = {2012},
    NUMBER = {5},
     PAGES = {449--460},
      ISSN = {0254-9409,1991-7139},
   MRCLASS = {65N30 (65N15)},
  MRNUMBER = {2988473},
       DOI = {10.4208/jcm.1203-m3855},
}

@article {MR1388486,
    AUTHOR = {Schatz, A. H. and Sloan, I. H. and Wahlbin, L. B.},
     TITLE = {Superconvergence in finite element methods and meshes that are
              locally symmetric with respect to a point},
   JOURNAL = {SIAM J. Numer. Anal.},
    VOLUME = {33},
      YEAR = {1996},
    NUMBER = {2},
     PAGES = {505--521},
      ISSN = {0036-1429},
   MRCLASS = {65N30 (65N12)},
  MRNUMBER = {1388486},
       DOI = {10.1137/0733027},
}

@article {MR4435940,
    AUTHOR = {Li, B.},
     TITLE = {Maximum-norm stability of the finite element method for the
              {N}eumann problem in nonconvex polygons with locally refined
              mesh},
   JOURNAL = {Math. Comp.},
    VOLUME = {91},
      YEAR = {2022},
    NUMBER = {336},
     PAGES = {1533--1585},
      ISSN = {0025-5718},
   MRCLASS = {65N30 (65N12)},
  MRNUMBER = {4435940},
MRREVIEWER = {Zhifeng Weng},
       DOI = {10.1090/mcom/3724},
}

@article {MR555385,
    AUTHOR = {Rannacher, R.},
     TITLE = {On nonconforming and mixed finite element methods for plate
              bending problems. {T}he linear case},
   JOURNAL = {RAIRO Anal. Num\'{e}r.},
    VOLUME = {13},
      YEAR = {1979},
    NUMBER = {4},
     PAGES = {369--387},
      ISSN = {0399-0516,0516-2777},
   MRCLASS = {65N30 (73K10)},
  MRNUMBER = {555385},
       DOI = {10.1051/m2an/1979130403691},
}

@book {MR775683,
    AUTHOR = {Grisvard, P.},
     TITLE = {Elliptic {P}roblems in {N}onsmooth {D}omains},
    SERIES = {Monographs and Studies in Mathematics},
    VOLUME = {24},
 PUBLISHER = {Pitman (Advanced Publishing Program), Boston, MA},
      YEAR = {1985},
     PAGES = {xiv+410},
      ISBN = {0-273-08647-2},
   MRCLASS = {35J25 (35-02)},
  MRNUMBER = {775683},
MRREVIEWER = {P.\ Szeptycki},
}

@article {MR645661,
    AUTHOR = {Rannacher, R. and Scott, R.},
     TITLE = {Some optimal error estimates for piecewise linear finite
              element approximations},
   JOURNAL = {Math. Comp.},
    VOLUME = {38},
      YEAR = {1982},
    NUMBER = {158},
     PAGES = {437--445},
      ISSN = {0025-5718,1088-6842},
   MRCLASS = {65N30},
  MRNUMBER = {645661},
       DOI = {10.2307/2007280},
}

@article {MR4050547,
    AUTHOR = {Apel, T. and Rogovs, S. and Pfefferer, J. and
              Winkler, M.},
     TITLE = {Maximum norm error estimates for {N}eumann boundary value
              problems on graded meshes},
   JOURNAL = {IMA J. Numer. Anal.},
    VOLUME = {40},
      YEAR = {2020},
    NUMBER = {1},
     PAGES = {474--497},
      ISSN = {0272-4979,1464-3642},
   MRCLASS = {65N15 (65N30)},
  MRNUMBER = {4050547},
       DOI = {10.1093/imanum/dry076},
}

@article {MR4748207,
    AUTHOR = {Diening, L. and Rolfes, J. and Salgado, A. J.},
     TITLE = {Pointwise gradient estimate of the {R}itz projection},
   JOURNAL = {SIAM J. Numer. Anal.},
  FJOURNAL = {SIAM Journal on Numerical Analysis},
    VOLUME = {62},
      YEAR = {2024},
    NUMBER = {3},
     PAGES = {1212--1225},
      ISSN = {0036-1429,1095-7170},
   MRCLASS = {65N30 (65N12 65N80)},
  MRNUMBER = {4748207},
       DOI = {10.1137/23M1571800},
}

@article {MR3498514,
    AUTHOR = {Leykekhman, D. and Vexler, B.},
     TITLE = {Pointwise best approximation results for {G}alerkin finite
              element solutions of parabolic problems},
   JOURNAL = {SIAM J. Numer. Anal.},
  FJOURNAL = {SIAM Journal on Numerical Analysis},
    VOLUME = {54},
      YEAR = {2016},
    NUMBER = {3},
     PAGES = {1365--1384},
      ISSN = {0036-1429,1095-7170},
   MRCLASS = {65M60 (65M15)},
  MRNUMBER = {3498514},
MRREVIEWER = {B\"{u}lent\ Karas\"{o}zen},
       DOI = {10.1137/15M103412X},
}

@article {MR1252774,
    AUTHOR = {Wang, M.},
     TITLE = {{$L^\infty$} convergence of conforming finite elements for the
              biharmonic equation},
   JOURNAL = {J. Comput. Math.},
  FJOURNAL = {Journal of Computational Mathematics},
    VOLUME = {11},
      YEAR = {1993},
    NUMBER = {4},
     PAGES = {373--384},
      ISSN = {0254-9409,1991-7139},
   MRCLASS = {65N12 (65N30)},
  MRNUMBER = {1252774},
MRREVIEWER = {Andris\ Buikis},
}

@article {MR614817,
    AUTHOR = {Maz'ya, V. G. and Plamenevski\u{\i}, B. A.},
     TITLE = {On the maximum principle for the biharmonic equation in a
              domain with conical points},
   JOURNAL = {Izv. Vyssh. Uchebn. Zaved. Mat.},
  FJOURNAL = {Izvestiya Vysshikh Uchebnykh Zavedeni\u{\i}. Matematika},
      YEAR = {1981},
    NUMBER = {2},
     PAGES = {52--59},
      ISSN = {0021-3446},
   MRCLASS = {35J30 (31B30 35B50)},
  MRNUMBER = {614817},
}

@article {MR1898745,
    AUTHOR = {Girault, V. and Scott, L. R.},
     TITLE = {Hermite interpolation of nonsmooth functions preserving
              boundary conditions},
   JOURNAL = {Math. Comp.},
  FJOURNAL = {Mathematics of Computation},
    VOLUME = {71},
      YEAR = {2002},
    NUMBER = {239},
     PAGES = {1043--1074},
      ISSN = {0025-5718,1088-6842},
   MRCLASS = {65D05 (65N15 65N30)},
  MRNUMBER = {1898745},
       DOI = {10.1090/S0025-5718-02-01446-1},
}

@article {MR1388488,
    AUTHOR = {Zhang, X.},
     TITLE = {Two-level {S}chwarz methods for the biharmonic problem
              discretized conforming {$C^1$} elements},
   JOURNAL = {SIAM J. Numer. Anal.},
  FJOURNAL = {SIAM Journal on Numerical Analysis},
    VOLUME = {33},
      YEAR = {1996},
    NUMBER = {2},
     PAGES = {555--570},
      ISSN = {0036-1429},
   MRCLASS = {65N30 (65F10 65N55)},
  MRNUMBER = {1388488},
       DOI = {10.1137/0733029},
}

@article {MR3925475,
    AUTHOR = {Kopteva, Natalia},
     TITLE = {Logarithm cannot be removed in maximum norm error estimates
              for linear finite elements in 3{D}},
   JOURNAL = {Math. Comp.},
  FJOURNAL = {Mathematics of Computation},
    VOLUME = {88},
      YEAR = {2019},
    NUMBER = {318},
     PAGES = {1527--1532},
      ISSN = {0025-5718,1088-6842},
   MRCLASS = {65N30 (65N15)},
  MRNUMBER = {3925475},
MRREVIEWER = {Liang\ Ge},
       DOI = {10.1090/mcom/3384},
}

@article {MR847015,
    AUTHOR = {Blum, H. and Lin, Q. and Rannacher, R.},
     TITLE = {Asymptotic error expansion and {R}ichardson extrapolation for
              linear finite elements},
   JOURNAL = {Numer. Math.},
  FJOURNAL = {Numerische Mathematik},
    VOLUME = {49},
      YEAR = {1986},
    NUMBER = {1},
     PAGES = {11--37},
      ISSN = {0029-599X,0945-3245},
   MRCLASS = {65N30 (65B05)},
  MRNUMBER = {847015},
MRREVIEWER = {Manfred\ Dobrowolski},
       DOI = {10.1007/BF01389427},
}

\end{document}